\documentclass{article}

\usepackage[authoryear]{natbib}
\usepackage[utf8]{inputenc}
\usepackage[english]{babel}
\usepackage{amsmath}
\usepackage{amsthm}
\usepackage{amssymb}
\usepackage{clrscode3e}
\usepackage[hidelinks]{hyperref}
\usepackage{float}
\usepackage{xspace}
\usepackage{xcolor}
\usepackage{soul}

\usepackage{todonotes}
\usepackage[a4paper,left=2.5cm,right=2.0cm,top=2.5cm, bottom=2.0cm]{geometry}
\usepackage{pdflscape}
\usepackage{afterpage}
\usepackage{rotating}
\usepackage{tikz}
\usepackage{pgfplots}
\usepackage{booktabs}

\usepackage{mathtools}

\usepackage{newfloat}
\DeclareFloatingEnvironment[
    fileext=los,
    listname={List Examples},
    name=Example,
    placement=tbhp,
]{example}

\usetikzlibrary{patterns, intersections, calc, decorations.pathmorphing,decorations.pathreplacing,shapes.geometric,positioning, arrows, shapes, decorations.markings,automata}
\pgfplotsset{compat=newest}

\usepgfplotslibrary{external} 
\usetikzlibrary{arrows}
\tikzset{
    > = {latex}
}

\theoremstyle{definition}
\newtheorem{lemma}{Lemma}
\newtheorem{theorem}{Theorem}
\newtheorem{corollary}{Corollary}
\newtheorem{definition}{Definition}

\newcommand{\I}{\mathcal{I}}
\newcommand{\mlaf}{\mathrm{AF}_{\mathrm{ml}}}
\newcommand{\caaf}{\mathrm{AF}_{\mathrm{ca}}}

\newcommand{\footremember}[2]{%
    \footnote{#2}
    \newcounter{#1}
    \setcounter{#1}{\value{footnote}}%
}
\newcommand{\footrecall}[1]{%
    \footnotemark[\value{#1}]%
} 

\title{Mathematical Models and Exact Algorithms \\for the Colored Bin Packing Problem}
\author{
  Yulle G. F. Borges\footnote{Corresponding Author, Email Address: \texttt{glebbyo@ic.unicamp.br}} \footremember{ic}{Institute of Computing, University of Campinas, 13083--852, Campinas--SP, Brazil.} \and 
  Rafael C. S. Schouery\footrecall{ic} \and 
  Fl\'avio K. Miyazawa\footrecall{ic}
}
\date{}

\begin{document}

\maketitle

\begin{abstract}
  This paper focuses on exact approaches for the Colored Bin Packing Problem (CBPP), a generalization of the classical one-dimensional Bin Packing Problem in which each item has, in addition to its length, a color, and no two items of the same color can appear consecutively in the same bin. To simplify modeling, we present a characterization of any feasible packing of this problem in a way that does not depend on its ordering.
  Furthermore, we present four exact algorithms for the CBPP\@.
  First, we propose a generalization of Valério de Carvalho's arc flow formulation for the CBPP using a graph with multiple layers, each representing a color.
  Second, we present an improved arc flow formulation that uses a more compact graph and has the same linear relaxation bound as the first formulation.
  And finally, we design two exponential set-partition models based on reductions to a generalized vehicle routing problem, which are solved by a branch-cut-and-price algorithm through VRPSolver.
  To compare the proposed algorithms, a varied benchmark set with $574$ instances of the CBPP is presented. Results show that the best model, our improved arc flow formulation, was able to solve over $62\%$ of the proposed instances to optimality, the largest of which with~$500$ items and $37$ colors. While being able to solve fewer instances in total, the set-partition models exceeded their arc flow counterparts in instances with a very small number of colors.

\end{abstract}

%\tarefadoyulle{Colocar no formato da revista}

\section{Introduction\label{sec-intro}}
The \textit{Bin Packing Problem} (BPP) consists of packing a set of items, each with a positive length, into the smallest possible number of bins such that the sum of the lengths of the items in a bin does not exceed its given capacity. A popular generalization of the BPP is the \textit{Cutting Stock Problem} (CSP), in which items have a positive demand value that indicates the number of times they must be included in a solution.
For generality, we use the CSP notation for some of the models presented in this work.

Applications of cutting and packing problems have been studied for decades. In the packing context, we have many problems in production planning and load balancing, in which one wants to fit a set of boxes into the smallest possible number of containers. On the other hand, the term cutting has been adopted in the industrial process of cutting rolls of stock material into smaller pieces in order to satisfy a given demand while minimizing the number of rolls used, which is often encountered in industries such as paper, metal, and fabric. These problems can also be encountered as part of bigger optimization problems found in real-world applications such as capacitated vehicle routing, load balancing scheduling, and capacitated facility location, among others.

The \textit{Colored Bin Packing Problem} (CBPP) is a generalization of the BPP in which each item, in addition to a length and demand, also has a color. The objective is to find a packing of all items in the smallest number of bins possible, in which no two items from the same color appear consecutively in the same bin. That is, for every bin in the solution there must exist at least one permutation of the items in it (considering all their respective copies in the case with demands) such that adjacent items always have different colors.

One possible application for the CBPP, in the particular case in which there are only two colors, is when there must be an intercalation of two types of items in a schedule. \citet{baloghBDKT13} describe the example in which white items can represent blocks of a radio show and black items represent commercial breaks, with the items' lengths indicating their duration. These blocks of content must be scheduled in an alternating form and respect the shows' duration, usually a time window of one hour.

When considering multiple colors, the same idea applies: CBPP can model the need to alternate items of different types. For example, \citet{dosaE14} describe an application where one wants to print batches of flyers on papers with different colors, and must alternate the colors to make it easy to split the print into the correct batches afterward. As another example, one can consider alternating different kinds of information and advertisement on the screen of a smartphone on sites like YouTube and other social media platforms, as described by~\citet{baloghBDKT13}.

Notice also that the development of good algorithms for CBPP could lead to the development of good algorithms for similar problems such as knapsack or scheduling problems where the items must alternate in some form. In fact, one can consider the CBPP as part of a larger class of packing problems where there are constraints between items on how or if they can be packed together.

Probably, the most famous problem in this class is the Bin Packing Problem with Conflicts, considered by \citet{jansenO97} and \citet{jansen99}, where a list of conflicts between pairs of items is given, and two items can be packed in the same bin only if there is no conflict between them. This problem appears, for example, as a sub-problem when solving BPP with a branch-and-price algorithm using a branching rule that either forces two items to be packed together or creates a conflict between them, as considered by \citet{vanceBJN94}.

Another problem in this class is the Class Constrained Bin Packing Problem \citep{shachnaiT01}, where every item belongs to some class, and there is a limit on the number of different classes that can be used in the same bin. One application for this problem is when there is a limited number of scissors in a machine for cutting metal coils, and the pieces can be grouped into classes of treatments to be received after the cut (for example, reducing thickness). Thus, the coils are cut into classes, each class receives its corresponding treatment, and then the final pieces are cut. This problem also has several applications in data load balance, such as Video-on-Demand servers \citep{xavierM08}. One main difference between the Class Constrained Bin Packing Problem, when compared with Bin Packing Problem with Conflicts, is that in the former, there is no direct conflict between items but, instead, what prevents us from adding a new item to a bin is the whole set of items already packed. Nonetheless, the ordering of the packing does not matter.

\citet{peetersD04} presented the Co-Printing Problem, a variant of the Class Constrained Bin Packing Problem where items have the same length, but can be of multiple classes at the same time. This is done to model a problem when printing Tetra Pak packages in a printer that can use a limited number of colors at once.

There are other packing problems where the order of the packing matters, such as Bin Packing with Largest in the Bottom Constraint proposed by \citet{manyemSV03}, where the largest items must be packed before the smaller ones in the bin. Notice that the difficulty of these problems usually arises when the items must be packed in an online fashion. The CBPP problem lies in a frontier between these kinds of problems. While the order of each packing matters for the online setting, in the offline setting, which is the focus of this paper, what matters is only the number of items of each color.

\subsection{Previous Works}
Previous works on CBPP are either from the perspective of approximation or online algorithms. These works consider three different versions of the problem: the \emph{offline} version, where the whole input is known at the beginning of the algorithm; the \emph{offline restricted} version, where the input is known at the beginning of the algorithm, but the items must be packed in the given order; and the \emph{online} version, where the items arrive in an online fashion and must be packed in a selected bin, without any chance of being changed afterward.

In 2012, \citet{baloghBDKT13} introduced the Black and White Bin Packing Problem (BWBPP) as a precursor of the CBPP in which only two colors are considered. They presented a lower bound of $1.7213$ on any online algorithm for the BWBPP, as well as a~$3$-competitive algorithm. They also give a $2.5$-approximation algorithm and an APTAS for the offline (unrestricted) version. These results were later published as two full articles in journals~\citep{baloghBDEKLT15,baloghBDEKT15}.

In 2014, \citet{dosaE14} introduced the CBPP\@. They proved that there is no optimal online algorithm for the zero-sized items variant of the CBPP, contrasting with BWBPP where such an algorithm exists. Their main results are a $4$-competitive online algorithm and a lower bound of $2$ on the asymptotic competitive ratio of any online algorithm. This last result also holds for the BWBPP, thus improving the previous result by \citet{baloghBDKT13}. \citet{chenHBT15} studied the special case of the BWBPP where items have lengths of at most half of the capacity of the bin, and presented a $8/3$-competitive algorithm for it, improving on the result of \citet{baloghBDKT13} which had a competitive ratio of $3$ even for this particular case.

Finally, \citet{bohmDESV18} considered the case of the CBPP where every item has length zero. They proved that the offline restricted optimum is always equal to the \textit{color discrepancy}, a measure of how unbalanced the colors in the instance are. Then, they provided an asymptotically $1.5$-competitive algorithm, which is, in fact, optimal as any deterministic online algorithm is shown to have at least this competitive ratio. The algorithm also has an absolute competitive ratio of $5/3$. For the general case, that is, items with arbitrary lengths and at least three colors, they designed a $3.5$-competitive algorithm and presented a lower bound of $2.5$ on the competitive ratio of any deterministic online algorithm for this problem. Interestingly, they also show that classical Bin Packing algorithms First Fit, Best Fit, and Worst Fit do not present a constant competitive ratio for the CBPP\@. For the BWBPP problem on the other hand, they prove that any fit algorithms (such as First Fit) have an absolute competitive ratio of $3$ and that any deterministic online algorithm has an asymptotic competitive ratio of $2$.

More recently, \citet{biloCMM18,biloCMM20} consider a game-theory version of the problem, where players control the items and decide where to pack them.

\subsection{Related Works}
In this section, we discuss some previous works that are directly connected to the contributions of this paper. The BPP and CSP have been the birthplace of some crucial development in mixed integer linear programming in the past. \citet{gilmoreG61} modeled the CSP with a set-covering model, in which columns represent all possible cutting configurations of a stock roll. To solve this model, they proposed the \textit{column generation} technique, in which cutting configurations are considered implicitly, that is, columns are only generated on an on-demand basis in an iterative process. This set-covering model was shown to have a very strong linear relaxation. Let $LP(I)$ be its linear relaxation value and $OPT(I)$ be the optimal value for instance $I$, the instance is said to have the \textit{Integer Round Up Property} (IRUP) if $OPT(I) = \lceil LP(I)\rceil$. Similarly, instance $I$ is said to have the \textit{Modified Integer Round Up Property} (MIRUP) if $OPT(I) - \lceil LP(I) \rceil \leq 1$. It is conjectured that MIRUP holds for all instances of the CSP and BPP~\citep{scheithauerT95}.

Later on, \citet{carvalho99} proposed the \textit{arc flow} model for the CSP based on a previous work of \citet{wolsey77} that connected integer programming with network flows. Given a directed graph where each vertex represents a cutting level of a stock roll, and arcs represent items, with its weight equal to the corresponding item length, a feasible cutting configuration can be found by computing a shortest path in such a graph. \citet{carvalho99} used this idea as the basis for a branch-and-price algorithm for the CSP and showed that the arc flow formulation is equivalent to Gilmore and Gomory's set-covering model in terms of their linear relaxation.

\citet{kramerIL21} approached a problem of job scheduling on parallel machines with family setup times, in which an additional setup time was required between two consecutive jobs of different families. They proposed an original arc flow formulation that uses one layer per family of jobs, modeling setup times as arcs crossing between layers. The authors also proposed another arc flow model that uses a single layer, but divides arcs into job-arcs, dummy-arcs which are zero-sized and allow for the transition between jobs from the same family, and setup-arcs which represent the setup time between two different families. The model requires that any path must alternate job-arcs with either a setup-arc or a dummy arc.

Recently, \citet{pessoaSUV20} proposed a generic branch-and-cut-and-price solver that included several computational optimizations and improvements that were developed recently in the vehicle routing literature. Their model is generic enough to be applied to many combinatorial optimization problems. This was achieved by designing a generic Vehicle Routing Problem (VRP) variant that can be modeled as a set-cover or set-partition problem in which columns can be generated by solving a Resource-Constrained Shortest Path Problem (RCSP). Additionally, several critical algorithmic improvements were generalized for this variant. Even though the model's priority is generality, the authors presented experimental results for the BPP that are very competitive with state-of-the-art problem-specific algorithms and models at the time.

\subsection{Our Contributions}

First, we provide a characterization of any feasible packing for the CBPP in a way that does not depend on its ordering, making algorithmic modeling much simpler. This characterization, along with other preliminary concepts and definitions are presented in Section~\ref{sec-preliminaries}. Afterward, we propose the pseudo-polynomial multilayered arc flow formulation, which generalizes Valério de Carvalho's classic CSP model for the CBPP\@. Furthermore, we propose the original color-alternating arc flow formulation which works with a much more compact graph while retaining the same strong linear relaxation. These pseudo-polynomial formulations are discussed in Section~\ref{sec-pseudo}. Later on, we propose two exponential set-partition models based on reductions to a generalized vehicle routing problem which are solved by a branch-cut-and-price algorithm through VRPSolver. These models are detailed in Section~\ref{sec-exp}. Finally, we propose a new benchmark set that includes uniformly randomly generated instances, instances adapted from the CSP/BPP literature, as well as randomly generated instances with a Zipf distribution of colors. This benchmark set, as well as the numerical experiments evaluating all the models introduced in this work, are presented in Section~\ref{sec-experiments}. In Section~\ref{sec-conclusions}, we recapitulate the all results and provide our final remarks.

\section{Preliminaries\label{sec-preliminaries}}
In this section, we provide a few definitions to contextualize our contributions. Along the text, for any $k \in \mathbb{Z}^+$, we denote $\{1, 2, \dots, k\}$ by $[k]$.

\paragraph{\textbf{Colored Bin Packing Problem (CBPP)}} An instance of the problem is composed of an integer bin capacity~$L > 0$, an integer number of colors $Q \geq 2$ and a set $\I = [m]$ of items, each with an integer length~$l_u > 0$, integer demand value $d_u > 0$ and color $c_u \in [Q]$ for all~$u \in \I$. Since each item has a demand, we consider that all pairs of length and color are unique in~$\I$. We also consider an auxiliary set $\I'$ composed of $d_u$ copies of each item~$u \in \mathcal{I}$. This way, a solution for CBPP is a partition $\mathcal{B}$ of $\I'$, such that for each part $B \in \mathcal{B}$, the total length of the items in~$B$ is at most $L$, and there exists at least one permutation of $B$ in which no two items of the same color appear consecutively. The objective is to find a solution that minimizes the number of parts. We consider that for any $S \subseteq \mathcal{I}$, $\{S_1, S_2, \dots, S_Q\}$ is a partition of $S$ by color, that is~$u \in S_q$ if and only if $u \in S$ and $c_u = q$.

Before proceeding, we introduce the following example to better illustrate the problem. Consider an instance with bin capacity equal to $8$ and items as described in the figure below.

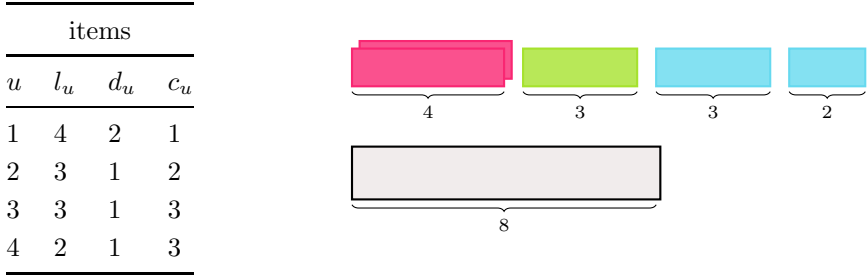
\begin{figure}[ht]
    \begin{minipage}{.4\textwidth}
            \centering
            \begin{tabular}{@{}llll@{}}
            \toprule
            \multicolumn{4}{c}{items}   \\ \midrule
            $u$ & $l_u$ & $d_u$ & $c_u$ \\ \midrule
            1   & 4     & 2     & 1     \\
            2   & 3     & 1     & 2     \\
            3   & 3     & 1     & 3     \\
            4   & 2     & 1     & 3     \\ \bottomrule
            \end{tabular}
    \end{minipage}%
    \begin{minipage}{.55\textwidth}
        \begin{tikzpicture}[scale=0.5]\tiny
            \definecolor{molokai_purple}{HTML}{AE81FF}
            \definecolor{molokai_yellow}{HTML}{E6DB74}
            \definecolor{molokai_red}   {HTML}{F92672}
            \definecolor{molokai_gray}  {HTML}{BCA3A3}
            \definecolor{molokai_blue}  {HTML}{66D9EF}
            \definecolor{molokai_green} {HTML}{A6E22E}
            \pgfmathsetmacro\buffer{0.2}
            %\draw[help lines] (0, -4) grid (11, 5);
                \begin{scope}[shift = ({0,8})]
                    \draw[thick, molokai_red, fill = molokai_red!80] (0+0.2, 0+0.2) rectangle ++(4, 1);
                    \draw[thick, molokai_red, fill = molokai_red!80] (0, 0) rectangle ++(4, 1);
                    \draw[decoration = {brace, mirror, raise = 2pt}, decorate] (0, 0) -- node[below = 5pt] {\scriptsize $4$} ++(4, 0);

                    \draw[thick, molokai_green, fill = molokai_green!80] (4.5, 0) rectangle ++(3, 1);
                    \draw[decoration = {brace, mirror, raise = 2pt}, decorate] (4.5, 0) -- node[below = 5pt] {\scriptsize $3$} ++(3, 0);

                    \draw[thick, molokai_blue, fill = molokai_blue!80] (8, 0) rectangle ++(3, 1);
                    \draw[decoration = {brace, mirror, raise = 2pt}, decorate] (8, 0) -- node[below = 5pt] {\scriptsize $3$} ++(3, 0);

                    \draw[thick, molokai_blue, fill = molokai_blue!80] (11.5, 0) rectangle ++(2, 1);

                    \draw[decoration = {brace, mirror, raise = 2pt}, decorate] (11.5, 0) -- node[below = 5pt] {\scriptsize $2$} ++(2, 0);

                    \draw[thick, fill = molokai_gray, fill opacity = 0.2] (0, -3) rectangle ++ (8+2*0.055, {1+2*\buffer});
                    \draw[decoration = {brace, mirror, raise = 2pt}, decorate] (0, -3) -- node[below = 5pt] {\scriptsize $8$} ++(8, 0);
                \end{scope}
        \end{tikzpicture}
    \end{minipage}
    \caption{An example of an instance of the CBPP with $4$ items, $3$ colors and bin capacity $8$.\label{example1}}
\end{figure}

An optimal solution for the instance in the example would be to pack the first copy of item $1$ with item $2$ in a single bin, and the second one with item $3$ in another bin. This would make it necessary to use a third bin to pack item $4$ by itself. This solution is illustrated in Figure~\ref{fig:sol1}. 

\begin{figure}[ht]
    \centering
    \begin{tikzpicture}[scale=0.5]\tiny
        \definecolor{molokai_purple}{HTML}{AE81FF}
        \definecolor{molokai_yellow}{HTML}{E6DB74}
        \definecolor{molokai_red}   {HTML}{F92672}
        \definecolor{molokai_gray}  {HTML}{BCA3A3}
        \definecolor{molokai_blue}  {HTML}{66D9EF}
        \definecolor{molokai_green} {HTML}{A6E22E}
        \pgfmathsetmacro\buffer{0.2}
        \begin{scope}[shift = ({0,0})]
            \draw[thick, fill = molokai_gray, fill opacity = 0.2] (0-2*\pgflinewidth,  2-\buffer) rectangle ++ (8+2*0.055, {1+2*\buffer});

            \draw[thick, fill = molokai_gray, fill opacity = 0.2] (0-2*\pgflinewidth,  0-\buffer) rectangle ++ (8+2*0.055, {1+2*\buffer});
            
            \draw[thick, fill = molokai_gray, fill opacity = 0.2] (0-2*\pgflinewidth, -2-\buffer) rectangle ++ (8+2*0.055, {1+2*\buffer});
        \end{scope}

        \begin{scope}[shift = ({0,0})]
            \draw[thick, molokai_red, fill = molokai_red!80] (0, 2) rectangle ++(4, 1);

            \draw[thick, molokai_red, fill = molokai_red!80] (0, 0) rectangle ++(4, 1);

            \draw[thick, molokai_green, fill = molokai_green!80] (4, 2) rectangle ++(3, 1);

            \draw[thick, molokai_blue, fill = molokai_blue!80] (4, 0) rectangle ++(3, 1);

            \draw[thick, molokai_blue, fill = molokai_blue!80] (0, -2) rectangle ++(2, 1);
        \end{scope}
    \end{tikzpicture}
    \caption{An optimal solution for the previous example.\label{fig:sol1}}
\end{figure}
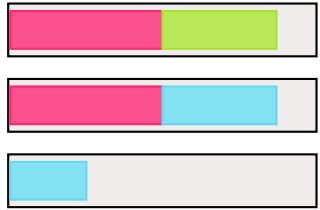

We call a \textit{packing pattern} any subsequence $P$ of items in $\I'$ such that the sum of the lengths of the items in $P$ is smaller than or equal to $L$ and, for $1 \leq k < |P|$, the color of the~$k$-th item of $P$ is different from that of the $(k+1)$-th one. With this, we may also describe a solution for the CBPP as the smallest set of packing patterns that satisfy the demands exactly for every item $u \in \mathcal{I}$. Note that, unlike in the CSP, a solution for the CBPP cannot contain more than $d_u$ copies of an item $u \in \mathcal{I}$, as otherwise two items of the same color could be intercalated with an additional differently colored item, possibly reducing the necessary number of bins.

In the following, we recall the important concept of \textit{color discrepancy} defined by~\cite{baloghBDEKT15} and generalized by~\cite{bohmDESV18}. Let $S$ be a fixed subset of $\I'$ and let $s_{q,u} = 1$ if item $u \in S$ is from color $q \in [Q]$ and $s_{q,u} = -1$ if item~$u$ is from any color other than $q$.  The \textit{discrepancy} between $q$-items and non-$q$-items in $S$ is $\delta^S_q = \sum_{u \in S} s_{q,u}$ for all $q \in [Q]$, and the \textit{color discrepancy} of~$S$ as~$\delta^S = \max_{q \in [Q]} \delta^S_q$. Finally, we define the \textit{critical color} of $S$ as $q_S^* = \arg\!\max_{q \in [Q]} \delta^S_q$ (choosing arbitrarily in case of a tie). We omit~$S$ from the notation in $\delta$ and $q^*$ unless it is necessary since it is usually clear by context. Theorem~\ref{theo:definition} gives us simple conditions to determine whether any subsequence of items admits a \emph{color-alternating permutation}, i.e.\ the items in it can be rearranged in such a way that no two items of the same color appear side-by-side. This allows us to verify, in linear time, if there exists a permutation of a given subset of items $S$ with $\sum_{u \in S} l_u \leq L$ that is a valid packing pattern. 

\begin{theorem}\label{theo:definition}
    Let $S \subseteq \I'$ and $\sum_{u \in S} l_u \leq L$. The following are equivalent:
    \begin{enumerate}
        \item $S$ admits a color-alternating permutation;
        \item $|S_{q^*}| \leq \lceil|S|/2\rceil$;
        \item $\delta \leq 1$.
    \end{enumerate}
\end{theorem}

The proof of this theorem follows from the lemmas below.

\begin{lemma}\label{lemma:necessary}
  If $S$ admits a color-alternating permutation, then $|S_{q^*}| \leq \lceil|S|/2\rceil$. 
\end{lemma}

\begin{proof}
  We prove that if~$|S_{q^*}| \geq \lceil|S|/2\rceil + 1$, then $S$ does not admit a color-alternating permutation. By the pigeonhole principle, since there are $|S_{q^*}|$ \mbox{$q^*$-items}, there must be at least $|S_{q^*}| - 1$ non-$q^*$-items for $S$ to admit a color-alternating permutation. Thus, there must be at least $2|S_{q^*}| - 1 \geq 2(\lceil|S|/2\rceil + 1) - 1 \geq |S| + 1$ items in total, a contradiction.
\end{proof}

\begin{lemma}\label{lemma:delta}
  If $|S_{q^*}| \leq \lceil|S|/2\rceil$, then $\delta \leq 1$.
\end{lemma}
\begin{proof}
  Since $\delta$ is the difference of the number of $q^*$-items and non-$q^*$-items in $S$, we can write it as $\delta = |S_{q^*}| - t$, where $t$ is the amount of non-$q^*$-items in $S$. Also, notice that if~$|S_{q^*}| \leq \lceil|S|/2\rceil$ holds, then $t \geq \lfloor|S|/2\rfloor$. Thus, we have
  \begin{align*}
      \delta = |S_{q^*}| - t 
             \leq |S_{q^*}| - \lfloor|S|/2\rfloor        
             \leq \lceil|S|/2\rceil - \lfloor|S|/2\rfloor 
             \leq 1.
           & \qedhere
  \end{align*}
\end{proof}

\begin{lemma}\label{lemma:sufficient}
  If $\delta \leq 1$, then $S$ admits a color-alternating permutation.
\end{lemma}

\begin{proof}
  For this, we give the algorithm \proc{Alternate} that receives $S\subseteq \I'$ such that $\delta \leq 1$ and returns a color-alternating permutation of $S$. We assume that the input set is partitioned by colors and can be accessed by color index. Also, we assume, without loss of generality, the items are non-decreasing ordered by color frequency.

  \begin{codebox}
      \Procname{$\proc{Alternate}(S,Q)$}
      \li Let $\mathcal{L}$ be a list of $|S|$ initially empty sequences
      \li $k \gets 0$
      \li \For $u \in S_1$\label{alg:firstloopb}
      \li     \Do
                  $k \gets k + 1$            \>\>\>\>\> \Comment{\scriptsize Counts amount of open sequences} 
      \li         $\mathcal{L}_k \gets \mathcal{L}_k \frown (u)$ \>\>\>\>\> \Comment{\scriptsize Operator $\frown$ means concatenation at the end of the sequence}
              \End\label{alg:firstloope}
      \li $j \gets 1$
      \li \For $q \in \{2,\dots,Q\}$\label{alg:secondloopb}
      \li     \Do
                  \For $u \in S_q$
      \li             \Do
                          $\mathcal{L}_j \gets \mathcal{L}_j \frown (u)$
                          \li                  \If $j == k$     \>\>\>\> \Comment{\scriptsize Places one item in each of the \id{k} sequences before returning to the first one}
      \li                     \Then 
                                  $j \gets 1$
      \li                     \Else
                                  $j \gets j + 1$
                              \End
                      \End 
              \End\label{alg:secondloope}
      %\li \Comment{Put sequences together to return}
      %\li $j \gets 0$, $p_j \gets \const{nil} \quad \forall j \in [|S|]$
      %\li \For $q \in [k]$ \label{alg:finalloopb}
      %\li     \Do
      %            \For $u \in \mathcal{L}_q$
      %\li             \Do
      %                    $j \gets j+1$
      %\li                 $p_j \gets u$
      %                \End

      %                \End \label{alg:finalloope}
      %\li \Return $p$
      \li \Return $\mathcal{L}_1 \frown \dots \frown \mathcal{L}_k$
  \end{codebox}

  First, we argue that the color which has the most items in $S$ is color $q^*$. For this, consider a color $q$ with $\delta_q$ and color $q'$ with $\delta_{q'}$. If $\delta_q > \delta_{q'}$ then 
  \begin{align*}
                   %\delta_q > \delta_{q'} \implies
      |S_q| - (|S| - |S_q|) > |S_{q'}| - (|S| - |S_{q'}|) \implies
               2|S_q| - |S| > 2|S_{q'}| - |S| \implies
                    |S_{q}| > |S_{q'}|.
  \end{align*}
  The first inequality follows from the fact that $\delta_q$ can be computed as the number of items from color $q$ minus the number of items from the other colors, which can be stated as~$|S| - |S_q|$. With this, we have that the colors are ordered by their discrepancies in the input. 
  
  Now we show that algorithm \proc{Alternate} is correct. In the first loop, lines~\ref{alg:firstloopb}-\ref{alg:firstloope}, each item from the first color is placed in a separate sequence, and variable \id{k} counts the number of sequences used. Notice that, since $\delta \leq 1$, there are enough remaining items to put in each of the \id{k} sequences except, possibly, the last one, which in turn would make the concatenation of these sequences a valid color-alternating permutation. 

  However, it would still be possible that there are more items from the following colors, so we alternate the items into $k$ sequences and return to the first sequence after an item is placed in the $k$-th one.  The main loop, in lines~\ref{alg:secondloopb}-\ref{alg:secondloope}, uses this method to place all items in one of the sequences. Now suppose that, with this method, one item is followed by another one from the same color $q$ in a single sequence, then there must be more than \id{k} items from color $q$, which means $\delta_q > \delta$ resulting in a contradiction. Thus, with this method, there will be no item followed by another of the same color in a single sequence.  
  
  Since each sequence only has one item from each color, and all sequences start with an item of color $1$ but end in an item with a different color (except, possibly the last one), the algorithm returns a concatenation of the sequences in order, which must be a valid color-alternating permutation of $S$.
\end{proof}

\section{Pseudo Polynomial Models\label{sec-pseudo}}
In this section, we propose two new pseudo-polynomial formulations for the CBPP\@. Before describing our contributions, we remind the reader about the arc flow formulation of~\citet{carvalho99} for the BPP\@. The concepts of this formulation are crucial to the understanding of the models we propose in this section.

In the arc flow model, we consider the directed graph $G = (V, A)$ as follows: $V$ consists of all possible packing points of a bin, i.e. $V = \{0, 1, \dots, L\}$; $A_{item}$ is the set of item-arcs, which connect every $i$ and $j$ such that there exists an item of length $j-i$, i.e.~${A_{item} = \{(i,j)\colon j - i = l_u}$ for~${0 \leq i < j \leq L}$, and $u \in \I \}$; $A_{loss}$ is the set of loss-arcs that connect every pair of adjacent vertices, i.e.\ $A_{loss} = \{(k,k+1)$ for $0 \leq k \leq L-1\}$; and $A = A_{item} \cup A_{loss}$. A path from $0$ to $L$ in $G$ corresponds to a valid packing pattern for an instance of the BPP, in which the presence of an item-arc $(i,j)$ in the path represents the packing of an item of length $j-i$, and the presence of loss-arcs represent empty spaces left in the bin. With this in mind, the BPP can be modeled as the problem of finding a set of paths in $G$ with the smallest cardinality that covers the demands for all items in the instance, which is achieved with the arc flow integer programming formulation of~\citet{carvalho99}.

\subsection{Multilayered Arc Flow}\label{sec-pseudo-mlaf}
A straightforward adaptation of the arc flow model for the CBPP is presented in this section. The idea is based on the construction of a graph in which the vertices are all the packing points repeated for each color. We can then add arcs that only connect vertices of different colors, thus enforcing that any paths must alternate colors.

Consider the digraph $G(V,A)$ as follows:
\begin{itemize}
	\item The set of vertices $V$ is partitioned as $\{V_0, V_1, V_2, \dots, V_Q\}$:
	      \begin{itemize}
		      \item $V_0$ is the set with a single source vertex $(0,0)$,
		      \item $V_q = \{(1,q), (2,q) \dots , (L,q)\}$ is the set of packing points for color $q$;
	      \end{itemize}
	\item The set of arcs $A$ is partitioned as $\{A_{source}, A_{item}, A_{loss}\}$:
	      \begin{itemize}
		      \item $A_{source} = \{\big((0,0),(j,q)\big): (j,q) \in V$ and $\exists u \in \mathcal{I}: l_u = j, c_u = q\}$,
		      \item $A_{item} = \{\big((i,q),(j,q')\big): (i,q) \in V, (j,q') \in V_{q'}$, $q \neq q'$ and  $\exists u \in \mathcal{I}: l_u = j-i, c_u = q'\}$,
		      \item $A_{loss} = \{\big((i,q),(L,q)\big) : (i,q),(L,q) \in V$ for any $q \in [Q]$  and $1 \leq i \leq L-1\}$.
	      \end{itemize}
\end{itemize}

With this, each color has $L$ vertices, resulting in $|V| = L Q + 1$ and $|A| = O({(LQ)}^2)$. An item-arc connects a vertex $(i,q) \in V$ to $(j,q') \in V$ where $q\neq q'$ and there is an item in~$\I$ with color $q'$ and length $j-i$, while a loss-arc connects any vertex in $V_q$ to the sink vertex of that color $(L,q)$ for any $q \in Q$. Hence, departing from any vertex, all arcs emanating from it either arrive at a vertex of a different color from the one it emanates from or arrive at the sink of that color.

Like with the arc flow formulation, we can model the CBPP as the problem of finding the smallest number of paths from $0$ to $L$ in $G$ such that the demand for each item is attended to equality. Formulation $\mlaf$ described below is an adaptation of the arc flow model of~\citet{carvalho99} to work with the graph and notation introduced in this section. In it, $x$ is a vector of integer variables that indicates the amount of flow passing through each arc, and~$z$ is the integer variable indicating the cumulative amount of flow passing through all arcs that emanate from the source vertex $V_0$.
	{
		\begin{alignat}{4}
			\quad \mlaf =\, & \omit\rlap{$\displaystyle \min z$} \label{obj:mlaf}                                                                                                                                                                                                                                                  \\
			                & \mbox{s.t.}                                         &  & \quad & \sum_{((i,q'), (j,q)) \in A} x_{iq'jq} - \sum_{((j,q), (k,q'')) \in A} x_{jqkq''} & = \begin{cases} -z\\0\\z\end{cases} &  & \begin{array}{ll} \mbox{for } j=0, q = 0 \\ \mbox{for } j\in[L-1], q \in [Q]\\ \mbox{for } j=L, q \in [Q]
			                                                                                                                                                                                                              \end{array}\label{constr:mlaf:flow} \\
			                &                                                     &  &       & \sum_{\substack{((i,q'), (j,q)) \in A                                                                                                                                                                                              \\ j=i+l_u, q = c_u, q' \neq q}} x_{iq'jq}                         & = d_u                               &  & \text{ for } u \in \mathcal{I}\label{constr:mlaf:demand}                                                           \\
			                &                                                     &  &       & z,x_{iq'jq}                                                                     & \in \mathbb{Z}^+                    &  & \phantom{l}\forall\, ((i,q'), (j,q)) \in A\label{constr:mlaf:integrality},
		\end{alignat}
	}

The objective~\eqref{obj:mlaf} is to minimize the number of paths (or patterns) used.  Constraints~\eqref{constr:mlaf:flow} guarantee the conservation of flow in all vertices of $G$, except in the source $(0,0)$ and sinks~$(L,q)$ for any $q \in [Q]$.  Constraints~\eqref{constr:mlaf:demand} ensure that the demands be satisfied to the equality for every item in $\I$. Finally, Constraints~\eqref{constr:mlaf:integrality} are the integrality constraints.

Since there are no arcs in $A$ that connect vertices of the same color, any path must intercalate vertices of different colors by using arcs representing items of different colors.  The model will find the smallest set of paths that satisfy the demands to equality, which will also be an optimal solution for the CBPP since all valid patterns can be represented as a path in~$G$.

\subsection{Color-Alternating Arc Flow}\label{sec-pseudo-caaf}

Although formulation $\mlaf$ correctly models the CBPP, it requires a large graph with up to $LQ + 1$ vertices and $O({(LQ)}^2)$ arcs. Because of this, we propose an alternative model that requires at most $L + 1$ vertices and $O(L^2Q)$ arcs, by using a multi-graph instead. We describe this model in this section.

Consider a directed multi-graph $G(V,A)$ in which arcs are labeled with a color, that is, an arc is a tuple $(i,j,q)$ with $i$ and $j$ being its start and end points, and $q$ its color. We consider the additional color $Q + 1$ to indicate the loss-arcs. The graph is defined as follows:
\begin{itemize}
	\item The set of vertices $V$ is the packing points $\{0,1,\dots,L\}$;
	\item The set of arcs $A$ is partitioned as  $\{A_1, A_2, \dots, A_Q, A_{Q+1} \}$:
	      \begin{itemize}
		      \item $A_q = \{(i,j,q): 0 \leq i \leq j \leq L, \exists u \in \mathcal{I}: l_u = j-i, c_u = q\}$ for $q \in [Q]$, and
		      \item $A_{Q+1} = \{(i,L,Q+1): 1 \leq i \leq L-1\}$.
	      \end{itemize}
\end{itemize}

With this, the number of vertices is $L+1$ and there is an item-arc of color $q \in [Q]$ between vertices $i \in V$ and $j \in V$ if there is an item of color $q$ and length $j-i$ in~$\I$. There is also a loss-arc connecting each vertex in $V\setminus\{0,L\}$ to $L$. However, to guarantee the color constraints in this graph we must specify that two arcs from the same color cannot be used consecutively in any path. For this, we propose formulation $\caaf$, in which $x$ is a vector of integer variables that indicate the amount of flow going through each arc and $z$ is the integer variable that indicates cumulative the amount of flow passing through all arcs that emanate from the source vertex $0$.
	{\small
		\begin{alignat}{4}
			\quad \caaf = & \omit\rlap{ $\min$ $\displaystyle z$} \label{obj:caaf}                                                                                                                                                                                                                                                         \\
			              & \mbox{s.t.}                                         &  & \quad & \sum_{(i,j,q) \in A} x_{ijq} - \sum_{(j,k,q') \in A} x_{jkq'}    & = \begin{cases} -z\\0\\z\end{cases} &  & \begin{array}{ll} \mbox{for } j=0,\\ \mbox{for } j=1,\dots,L-1,\\ \mbox{for } j=L, \end{array}\label{constr:caaf:flow-conserv} \\
			              &                                                     &  & \quad & \sum_{(i,j,q) \in A} x_{ijq} - \sum_{\substack{(j,k,q') \in A:                                                                                                                                                                           \\ q'\in [Q+1]\setminus \{q\}}} x_{jkq'} & \leq 0                              &  & \text{ for } j=1,\dots,L-1, q \in [Q]\label{constr:caaf:color-alt} \\
			              &                                                     &  &       & \sum_{\substack{(i,j,q) \in A                                                                                                                                                                                                            \\ j=i+l_u,c_u=q}} x_{ijq}                             & = d_u                               &  & \text{ for } u \in \mathcal{I}_q, q \in [Q]\label{constr:caaf:demand} \\
			              &                                                     &  &       & z,x_{ijq}                                                        & \in \mathbb{Z}^+                    &  & \phantom{l}\forall\, (i,j,q) \in A, q \in [Q]\label{constr:caaf:integrality}
		\end{alignat}
	}

The objective~\eqref{obj:caaf} is to minimize the number of paths used. Constraints~\eqref{constr:caaf:flow-conserv} guarantee the conservation of flow in all vertices of $V$. Constraints~\eqref{constr:caaf:color-alt} ensure that the amount of flow passing through arcs of color $q\in[Q]$ that arrive in a vertex $j\in V$ cannot be greater than the amount of flow emanating from $j$ through arcs of any color other than $q$. Constraints~\eqref{constr:caaf:demand} ensure that the demands be satisfied to equality for every item in $\I$. Finally, Constraints~\eqref{constr:caaf:integrality} are the integrality constraints.

We claim that Constraints~\eqref{constr:caaf:color-alt} guarantee the color constraints of the CBPP\@. The idea is that if they are satisfied for a vertex $j\in V$, then there must exist a decomposition of the solution into paths such that any path that passes through $j$ arrives from an arc of a different color than the one it leaves from. In the remainder of this section, we show that, for each integer point in the polytope defined by $\caaf$, there exists a feasible solution to the CBPP associated with it. We enunciate Lemmas~\ref{lemma:tightness} and~\ref{lemma:non-blocked}, which help achieve this result with Lemma~\ref{lemma:decompose}. Furthermore, we show that any CBPP solution can be mapped into an integer point of $\caaf$ with the same objective value, in Lemma~\ref{lemma:decompose-back}.

\begin{lemma}\label{lemma:tightness}
	If any feasible solution $(\bar x, \bar z)$ of $\caaf$ uses arcs from at least two different colors, say $a$ and $b$, and satisfy Constraints~\eqref{constr:caaf:color-alt} to equality for both these colors at vertex ${j \in V \setminus \{0,L\}}$, then the only arcs $(i,j,q) \in A$ with $\bar x_{ijq} > 0$ have $q \in \{a,b\}$.
\end{lemma}

\begin{proof}
	Since, in vertex $j$, $\bar x$ satisfy Constraints~\eqref{constr:caaf:color-alt} to equality for $a$ and $b$ we have that
	\begin{align}
		\sum_{(i,j,a)\in A} \bar x_{ija} = & \sum_{\substack{(j,k,q) \in A \\ q \in [Q + 1]\setminus \{a\}}} \bar x_{jkq} = \sum_{(j,k,q) \in A} \bar x_{jkq} - \sum_{(j,k,a) \in A} \bar x_{jka}  \label{eq:1}, \\
		\sum_{(i,j,b)\in A} \bar x_{ijb} = & \sum_{\substack{(j,k,q) \in A \\ q \in [Q + 1]\setminus\{b\}}} \bar x_{jkq} = \sum_{(j,k,q) \in A} \bar x_{jkq} - \sum_{(j,k,b) \in A} \bar x_{jkb}\label{eq:2}.
	\end{align}
	Now, notice that
	\begin{align}
		\sum_{(i,j,q) \in A} \bar x_{ijq} & \geq \sum_{(i,j,a)\in A} \bar x_{ija} + \sum_{(i,j,b)\in A} \bar x_{ijb}                                      \\
		                                  & = \sum_{(j,k,q) \in A} \bar x_{jkq} - \sum_{(j,k,a) \in A} \bar x_{jka} +
		\sum_{(j,k,q) \in A} \bar x_{jkq} - \sum_{(j,k,b) \in A} \bar x_{jkb}                                                                             \\
		                                  & = 2 \sum_{(j,k,q) \in A} \bar x_{jkq} - \sum_{(j,k,a) \in A} \bar x_{jka} - \sum_{(j,k,b) \in A} \bar x_{jkb} \\
		                                  & \geq 2\sum_{(j,k,q) \in A} \bar x_{jkq} - \sum_{(j,k,q) \in A} \bar x_{jkq}                                   \\
		                                  & =\sum_{(j,k,q) \in A} \bar x_{jkq}.
	\end{align}
	Since any feasible solution must satisfy the flow conservation imposed by Constraints~\eqref{constr:caaf:flow-conserv}, both inequalities above are, in fact, equalities. Thus, there cannot exist arcs arriving at or emanating from $j$ from any color other than $a$ or~$b$ with positive value in $\bar x$.
\end{proof}

\begin{corollary}\label{corollary:tightness}
	For any feasible solution $(\bar x, \bar z)$ of $\caaf$ and vertex $j \in V \setminus \{0,L\}$, the number of colors that satisfy Constraints~\eqref{constr:caaf:color-alt} to equality at~$j$ and $\sum_{(i,j,q) \in A} \bar x_{ijq} > 0$ is at most $2$.
\end{corollary}

We show that if a solution $(\bar x, \bar z)$ is feasible, then there must exist a way to extract one \textit{color-alternating} path from $0$ to $L$ from $(\bar x, \bar z)$ and keep the remaining solution feasible. Color-alternating paths are those in which no arcs of the same color appear consecutively. The following definition will help us understand when the removal of a specific pair of arcs from $(\bar x, \bar z)$ will result in an infeasible solution.

\begin{definition}
	Given a feasible solution $(\bar x, \bar z)$ of $\caaf$, a vertex $j \in V\setminus \{0,L\}$, and an arc $(i,j,a)$, an arc $(j, k, b)$ is \textit{blocked} for $(i,j,a)$ if $a = b$ or there is a color $q \notin\{a, b\}$ such that \[\sum_{(i',j,q) \in A} \bar x_{i'jq} = \sum_{\substack{(j,k',q') \in A\\q' \in [Q + 1] \setminus\{q\}}} \bar x_{jk'q'}.\]
\end{definition}

\begin{lemma}\label{lemma:non-blocked}
	Given a feasible solution $(\bar x, \bar z)$ of $\caaf$, a vertex $j \in V\setminus \{0,L\}$, and an arc $(i,j,a)$ such that $\bar x_{ija} > 0$, there exists at least one unblocked arc $(j,k,b)$ with $\bar x_{jkb} > 0$.
\end{lemma}

\begin{proof}
	Since $\bar x_{ija} > 0$, then, by Corollary~\ref{corollary:tightness}, the number of colors that satisfy Constraints~\eqref{constr:caaf:color-alt} to equality for vertex $j$ is either~$0$,~$1$, or $2$. Thus, we have the following cases.

	If there are no colors that satisfy Constraints~\eqref{constr:caaf:color-alt} to equality, then, by Constraint~\eqref{constr:caaf:color-alt}, there is an arc $(j,k,b)$ with $b \neq a$ with $\bar x_{jkb} > 0$, which is unblocked.

	In case there is exactly one color $c$ that satisfies Constraint~\eqref{constr:caaf:color-alt} to equality,  if $c=a$, then any arc $(j, k, b)$ with $b \neq a$ is unblocked and, by Constraint~\eqref{constr:caaf:color-alt}, there must be such an arc with $\bar x_{jkb} > 0$. And, if $c \neq a$, then,  by Constraint~\eqref{constr:caaf:flow-conserv},
	\[
		\sum_{(i', j, a) \in A} \bar x_{i'ja} +
		\sum_{(i', j, c) \in A} \bar x_{i'jc}
		\leq
		\sum_{(i',j,q) \in A} \bar x_{i'jq}
		=
		\sum_{\substack{(j,k',q') \in A\\ q' \in [Q + 1] \setminus \{c\}}} \bar x_{jk'q'} +
		\sum_{(j,k',c) \in A} \bar x_{jk'c},
	\]
	thus as $c$ satisfies Constraint~\eqref{constr:caaf:color-alt} to equality and $\bar x_{ija} > 0$, there is an unblocked arc $(j,k,c)$ such that $\bar x_{jkc} > 0$.

	Lastly, if there are two colors, $c$ and $c'$, that satisfy Constraints~\eqref{constr:caaf:color-alt} to equality, then by Lemma~\ref{lemma:tightness} these are the only colors with arcs of positive flow in $j$. Therefore, without loss of generality, $c = a$ and, by Constraint~\eqref{constr:caaf:flow-conserv},
	\[
		\sum_{(i', j, a) \in A} \bar x_{i'ja} +
		\sum_{(i', j, c') \in A} \bar x_{i'jc'}
		=
		\sum_{(i',j,q) \in A} \bar x_{i'jq}
		=
		\sum_{\substack{(j,k',q') \in A\\ q' \in [Q+1]\setminus\{c\}}} \bar x_{jk'q'} +
		\sum_{(j,k',c') \in A} \bar x_{jk'c'},
	\]
	from which we conclude that there is an unblocked arc $(j,k,c')$ such that $\bar x_{jkc'} > 0$.
\end{proof}

With Lemmas~\ref{lemma:tightness} and~\ref{lemma:non-blocked} in place, we can describe an algorithm to decompose a feasible solution into color-alternating paths. Algorithm \proc{DecomposeAF} takes an integer feasible solution $\bar x$, and the instance graph $G(V,A)$ as input and produces a set of $z$ color-alternating paths.

\begin{codebox}
	\Procname{$\proc{DecomposeAF}(\bar x, \bar z, V, A)$}
	\li \If $\bar z = 0$\label{stop-condition}
	\li     \Then
	\Return $\emptyset$
	\End
	\li $\id{in} \gets $ arbitrary arc $(0,j,a) \in A$ with $\bar x_{0ja} > 0$
	\li $P \gets \{\id{in}\}$
	%\li $\bar x_{0ja} \gets \bar x_{0ja} -1$ \label{decrement:1}
	\li \While $j \neq L$\label{main-loop:begin}
	\li     \Do
	%            \For $q \gets 1$ \To $Q$ \label{degree-loop:begin}
	%\li             \Do
	%                    $\id{in}(q)  \gets $ sum of $\bar x_{jk}$ for all arcs $(i,j) \in A$ from color $q$ 
	%\li                 $\id{out}(q) \gets $ sum of $\bar x_{ij}$ for all arcs $(j,k) \in A$ from color $q$ 
	%                \End \label{degree-loop:end}
	\For each arc $\id{out} = (j,k,b) \in A$ with $\bar x_{jkb}>0$ and $b\neq a$\label{choice-loop:begin}
	\li             \Do
	\For $q \gets 1$ \To $Q$ and $q \neq b$
	\Do
	%\li                         \If $\id{in}(q) = \func{sum}(\id{out}) - \id{out}(q)$ \>\>\>\>\>\> \Comment Notice that $\id{in}(c(a))$ was decremented by 1 \label{block-condition}
	%\li                         \If $\id{in}(q) > \func{sum}(\id{out}) - \id{out}(q) - 1 $ \>\>\>\>\>\> \Comment Notice that $\id{in}(c(a))$ was decremented by 1
	\li                         \If $\sum_{(i,j,q) \in A} \bar x_{ijq} = \sum_{(j,k,q') \in A:q'\neq q} \bar x_{jkq'}$
	%\>\>\>\>\>\> \scriptsize \Comment Note that the flow in $\id{in}$ was decremented  \label{block-condition}
	\li                             \Then
	block arc $\id{out}$ \>\>\>\>\> \scriptsize\Comment Arc $\id{out}$ is blocked by color $q$
	\End
	\End
	\li                     \If arc $\id{out}$ is not blocked
	\li                         \Then
	$P \gets P \cup \{\id{out}\}$
	%\li                             $\bar x_{jkb} \gets \bar x_{jkb} -1$ \label{decrement:2}
	\li                             $\id{in} \gets \id{out}$, $a \gets b$, $j \gets k$
	\li                             \textbf{break}
	\End
	\End\label{choice-loop:end}
	\End\label{main-loop:end}
	\li $\check d \gets \min_{(i,j,q) \in p} \bar x_{ijq}$\label{decrement-step}
	\li \For each arc $(i,j,q) \in P$\label{decrement:begin}
	\li     \Do
	$\bar x_{ijq} \gets \bar x_{ijq} - \check d$
	\li         $\bar z \gets \bar z - \check d$
	\End\label{decrement:end}
	\li \Return $\{(\check d,P)\} \cup \proc{DecomposeAF}(\bar x, \bar z, V, A)$\label{recursive-call}
\end{codebox}

The idea is that with each call of \proc{DecomposeAF} for a  solution $(\bar x, \bar z)$ of $\caaf$, one color-alternating path is removed from $\bar x$ and added to the output while also maintaining feasibility of the remaining solution. Then, by induction, the same algorithm can be repeated to the remaining solution until there are no arcs with positive flow left.

\begin{lemma}\label{lemma:decompose}
	Algorithm \proc{DecomposeAF} decomposes a solution $\bar x$ of the model $\caaf$ into color-alternating paths from $0$ to $L$ in $G(V,A)$ with flow equal to $\bar z$.
\end{lemma}

\begin{proof}
	Condition in Line~\ref{stop-condition} is the base case, for when there is no positive flow in the solution.

	The main loop, in Lines~\ref{main-loop:begin}-\ref{main-loop:end}, iterates over $V$ starting from the endpoint of the arc chosen in the previous iteration until it reaches $L$. Loop in Lines~\ref{choice-loop:begin}-\ref{choice-loop:end} iterates over all arcs with positive flow and color different from $a$ to find an unblocked arc $\id{out}$ with positive flow. From Lemma~\ref{lemma:non-blocked}, we know that this loop will always find such an arc, which is then added to the current path and the loop continues from its endpoint.
	%Note that condition~\eqref{block-condition} is sufficient to check if color $q$ blocks the pair of arcs $(\id{in},\id{out})$ since \id{in} had its flow value decremented by one (lines~\eqref{decrement:1} and~\eqref{decrement:2}) in a previous iteration. 
	Line~\ref{decrement-step} computes the smallest amount of flow $\check d$, passing through the arcs in the recently formed path. This is the amount of flow that can be removed from all arcs in the path while maintaining the remaining solution feasible and setting at least one variable of the new solution $(\bar x, \bar z)$ to zero. The loop in Lines~\ref{decrement:begin}-\ref{decrement:end} decrements exactly $\check d$ from all variables representing arcs in the recently formed path. At last, in Line~\ref{recursive-call}, the current path and its multiplicity are concatenated with the result of a recursive call of the algorithm on the remaining solution.

	Notice that before the recursive call, the same amount of flow is removed from a set of arcs connecting $0$ to $L$, which in turn keeps Constraints~\eqref{constr:caaf:flow-conserv} satisfied. Since none of the removed pairs of arcs are blocked by any color, the remaining solution before any recursive call to \proc{DecomposeAF} is always feasible. Furthermore, since the amount of flow removed from the arcs in each iteration is guaranteed to set at least one variable to zero, this procedure is guaranteed to eventually stop (in at most $O(|A|)$ steps).

	Notice that, adding $\check d$ for every path $P$ the value obtained is precisely $\bar z$.
\end{proof}

\begin{lemma}\label{lemma:decompose-back}
	Any feasible solution to the CBPP that uses $\bar z$ bins can be represented as a feasible solution of model $\caaf$ with objective value $\bar z$.
\end{lemma}

\begin{proof}
	Consider a feasible solution to the CBPP with objective value $\bar z$. This solution can be described as $\bar z$ packing patterns. We must show that any packing pattern~$P$ can be represented as a color-alternating path from $0$ to $L$. For this, we select the first item~$u$ of~$P$ (respecting the given order) and add one unit of flow to the arc corresponding to that item, i.e.\ the arc of length $l_u$ and color $c_u$ emanating from $0$, and repeat the process starting from vertex $0 + l_u$ with sequence $P \setminus \{u\}$. This procedure stops either when the end of the sequence~$P$ or vertex $L$ are reached, with the former being completed with a loss-arc emanating from the current vertex and arriving in $L$.

	Since $P$ is a packing pattern, at any vertex $j \in V$ there is exactly one unit of flow arriving in $j$ from an arc of color $q$ and one unit of flow emanating from $j$ from an arc of color $q' \in [Q] \setminus \{q\}$, thus Constraints~\eqref{constr:caaf:color-alt} are satisfied to equality for color $q$ in $j$. This procedure is done for each packing pattern in the solution, and the resulting paths are overlapped to form the solution to model $\caaf$. Note that when all paths are overlapped, although Constraints~\eqref{constr:caaf:color-alt} might not be satisfied to equality because the same color $q'$ can be used for different incoming colors, they still hold. Lastly, for each feasible pattern, one unit of flow is added only to a set of arcs that connect $0$ to $L$, so Constraints~\eqref{constr:caaf:flow-conserv} are always satisfied.
\end{proof}

With the results of Lemmas~\ref{lemma:decompose} and~\ref{lemma:decompose-back}, we can finally state the following Theorem.

\begin{theorem}
	The arc flow formulation~$\caaf$ models the CBPP\@.
\end{theorem}

\subsection{Comparing $AF_{ml}$ and $AF_{ca}$}
Although they differ in graph representation, both models proposed in this section use the same general idea of arc flow. These types of models typically present very strong lower bounds for bin packing problems. In this section, we develop a correlation between the linear programming relaxation of both proposed models.

First, we remark that model $AF_{ml}$ is similar to the formulation proposed by \citet{carvalho99} for the BPP, although modeled on a different graph to incorporate the color constraints into the flow paths. The latter formulation has been shown to provide the same linear programming lower bound as the set-covering model of \citet{gilmoreG61} for the BPP\@, which is the basis for the MIRUP conjecture, that states that the optimal value of any instance is within one plus this lower bound rounded up~\citep{scheithauerT95}.

In our case, the lower bound provided by $AF_{ml}$ linear relaxation is the same as the lower bound provided by the linear relaxation of the set-partition model (analogous to~\citet{gilmoreG61}) in which all color constraints are handled exclusively in the pricing subproblem.

\begin{lemma}\label{lemma:ml-to-ca}
	Given a feasible solution $(\bar x, \bar z)$ of the linear relaxation of $AF_{ml}$, it is possible to compute a corresponding solution $(\bar x', \bar z)$ that is feasible for the linear relaxation of the $AF_{ca}$.
\end{lemma}

\begin{proof}
	In this proof, we denote the set of arcs of the graph corresponding to formulation $AF_{ca}$ by $A'$.

	For all $(i, j, q) \in A'$, let us define
	\[
		\bar x'_{ijq} = 
		\begin{cases}
			\displaystyle\sum_{\substack{((i, q'), (j, q)) \in A\\q' \neq q}} \bar x_{iq'jq}, & \text{if } q \leq Q,\\
			\displaystyle\sum_{q' \in [Q]}\sum_{\substack{((i, q'), (j, q')) \in A}} \bar x_{iq'jq'}, & \text{otherwise,}\\
		\end{cases}
	\]
	where the second case corresponds to the loss-arcs, which only exist for $j = L$. We prove that $(\bar x', \bar z)$ is a feasible solution for $AF_{ca}$.
	
	Since~\eqref{constr:mlaf:demand} is feasible in $\bar x$, for all $u \in \I$ and considering $q = c_u$, we have that
	\[
		\sum_{\substack{(i, j, q') \in A\\j=i+l_u, q'=c_u}} \bar x'_{ijq'} 
		% = 
		% \sum_{\substack{(i, j, q) \in A\\j=i+l_u}} \bar x'_{ijq} 
		= 
		\sum_{\substack{(i, j, q) \in A\\j=i+l_u}}
		\sum_{\substack{q' \neq q\colon\\ ((i, q'), (j, q)) \in A}} \bar x_{iq'jq}
		=
		\sum_{\substack{((i,q'),(j,q)) \in A\colon\\ j = i + l_u, q' \neq q}} \bar x_{iq'jq}
		= d_u,
	\]
	and, thus, we conclude that Constraints~\eqref{constr:caaf:demand} are valid for $\bar x'$.

	Now, we can show that Constraints~\eqref{constr:caaf:flow-conserv} hold, as $\bar x$ satisfies Constraints~\eqref{constr:mlaf:flow}. In fact, for all $j \in \{0, \dots, L\}$, we have that
	\begin{align*}
		\sum_{(i,j,q) \in A'} \bar x_{ijq}' - \sum_{(j,k,q') \in A'} \bar x_{jkq'}'
		 & = \sum_{((i, q'), (j, q)) \in A} \bar x_{iq'jq} -
		\sum_{((j, q), (k, q'')) \in A} \bar x_{jqkq''}      \\
		 & = \begin{cases}
			     -z & \mbox{for } j=0, q = 0         \\
			     0  & j\in\{1,\dots,L-1\}, q \in [Q] \\
			     z  & j=L, q \in [Q].                \\
		     \end{cases}
	\end{align*}

	Finally, we can derive the validity of Constraints~\eqref{constr:caaf:color-alt} as follows. For fixed ${j \in \{0, \dots, L-1\}}$ and $q \in [Q]$, we have that
	\begin{align}
		\sum_{(i,j,q) \in A'} \bar x'_{ijq} - \sum_{\substack{(j,k,q') \in A' \\ q' \in [Q + 1] \setminus \{q\}}} \bar x'_{jkq'}
		 & =
		\sum_{((i,q'),(j,q)) \in A} \bar x_{iq'jq} -
		\sum_{((j, q''), (k, q')) \in A} \bar x_{jq''kq'}\label{eq:ca:no_equal}\\
		 & \leq
		\sum_{((i,q'),(j,q)) \in A} \bar x_{iq'jq} -
		\sum_{((j,q),(k,q')) \in A} \bar x_{jqkq'} \label{eq:ca:equality}       \\
		 & = 0,\label{eq:ca:flow}
	\end{align}
	where Equation~\eqref{eq:ca:no_equal} follows from the fact that there are no loss-arcs with $j < L$, Inequality~\eqref{eq:ca:equality} holds as $\bar x \geq 0$ and  Equation~\eqref{eq:ca:flow} follows from the flow conservation imposed by Constraint~\eqref{constr:mlaf:flow}.
\end{proof}

\begin{lemma}\label{lemma:ca-to-ml}
	Given a solution $(\bar x, \bar z)$ to the linear programming relaxation of $AF_{ca}$, it is possible to compute a corresponding solution $(\bar x', \bar z)$ that is feasible for the linear relaxation of $AF_{ml}$.
\end{lemma}

\begin{proof}
	Consider that we run algorithm \proc{DecomposeAF} with solution $(\bar x, \bar z)$ and obtain a set of color-alternating patterns $\mathcal{P}$.
	% As $(\bar x, \bar z)$ is not necessarily integral, 

	Now, for each path $P = ((i_0, i_1, q_1), (i_1, i_2, q_2), \dots, (i_{k-1},i_k,q_k))$ with $i_0 = 0$ generated from $\hat x$, we define a path $P'$ for $AF_{ml}$ such that \[P' = (((i_0,q_0), (i_1,q_1)), ((i_1,q_1),(i_2,q_2)), \dots, ((i_{k-1},q_{k-1}), (i_k,q')))\] where $q_0 = 0$ and $q' = q_{k-1}$ if $(i_{k-1},i_k,q_k)$ is a loss-arc and $q'= q_{k}$ otherwise. Finally, we add~$\check d$ units of flow to each arc in $P'$, where $\check d$ is the amount of flow passing through the path~$P$ in $\hat x$, and we obtain a solution $(\bar x', \bar z)$ of $AF_{ml}$ by combining the flow of every path in $\mathcal{P}$.

	It is easy to see that, by construction, $\bar x'$ is non-negative and satisfies the flow conservation constraints and the demand constraints.
\end{proof}

As a consequence of Lemmas~\ref{lemma:ml-to-ca}~and~\ref{lemma:ca-to-ml}, we can state the following theorem about the strength of the linear relaxation of the proposed models.

\begin{theorem}
	The lower bounds provided by the linear relaxation of formulations $AF_{ml}$ and $AF_{ca}$ are the same.
\end{theorem}

\subsection{Graph Reduction}
Both models presented in this section represent packing patterns as paths in a digraph with $L + 1$ vertices. This makes the number of variables (and constraints) of the models grow pseudo-polinomially with the bin capacity. To mitigate this, we use a reduction technique known as \textit{normal patterns} (also known as \textit{canonical dissections}). This technique was proposed independently by \citet{herz72} and \citet{christofidesW77}, and it reduces the number of packing points of a bin by showing that there will always exist an optimal solution in which all items are packed in the leftmost position available. The points of a bin in which it is possible to pack an item are defined as $\{x = \sum_{i \in \I} l_i \pi_i: 0\leq x \leq L, \pi_i \in \{0,1\}$, for $i \in \I\}$. This set is computed through a dynamic programming algorithm \citep{coteI18} before building the graph. The graph is then created using only vertices that represent points in the normal patterns, thus achieving a more compact graph while still preserving optimality.

Normal patterns is one of the most generic reduction techniques for packing points in a bin. Although there exist other more aggressive reductions, such as meet-in-the-middle proposed by~\citet{coteI18} and reflect proposed by~\citet{delormeI20} (specifically for arc flow), those require extra investigative work to be adapted for our problem since the order of the packing matters.

\section{Exponential Models\label{sec-exp}}
In this section, we propose set-partition exponential models for the CBPP that represent all packing patterns implicitly. These models can be solved through branch-and-price or branch-cut-and-price algorithms, and usually present very strong upper bounds~\citep{delormeIM16}. We use the VRPSolver generic scheme, proposed by \citet{pessoaSUV19, pessoaSUV20} to design such models for the CBPP\@. With this scheme, we reduce our problem into a variant of the Vehicle Routing Problem (VRP), which is then solved by a black-box branch-and-cut-and-price algorithm. 

In the remainder of this section, we describe how the VRPSolver works in the context of the CBPP, and then describe the two models proposed. 

\subsection{VRPSolver}
VRPSolver is a generic branch-cut-and-price solver that incorporates several modern algorithmic improvements such as rank-1 cuts with limited memory, path enumeration, and rounded capacity cuts. It solves a variant of the VRP that is sufficiently general to model many other variants of the VRP or even other problems such as the BPP and the Generalized Assignment Problem. A model for the BPP using this framework proposed by \citet{pessoaSUV20} showed remarkable results when compared to other modern algorithms designed specifically for the BPP\@. While we do not go into detail on how the framework manages to generalize the key elements of branch-cut-and-price algorithm, we strongly recommend the paper by \citet{pessoaSUV20} to the interested reader.  

The main concept of VRPSolver is to provide a formulation that can be solved by a branch-cut-and-price algorithm in which the pricing problem is modeled as a Resource Constrained Shortest Path problem (RCSP). The variables of the resulting paths are linked to specifically designed objective function and constraints through \textit{mappings}. These constraints and objective function are then added to the original master formulation to be solved by the branch-and-cut-and-price algorithm.

\subsubsection{Resource Constrained Shortest Path Problem}\label{sec:RCSP}
An instance of the RCSP is composed of a directed multi-graph $G = (V,A)$, a source vertex $v_{source} \in V$, a sink vertex $v_{sink} \in V$, and a set of resources $R$. For each arc $a \in A$, there is a cost $c_a$ and, for each resource $r \in R$, there is a consumption value $q_{ar}\in \mathbb{R}$. There is a finitely accumulated resource consumption interval $[l_{ar}, u_{ar}]$ for each resource $r\in R$ and arc~$a~\in~A$. 
%
%We may assume that either $G$ is acyclic or there is at least one resource with no negative consumption.
%
A path $p=(v_{source}, a_1, v_1, \dots, a_{n-1}, v_{n-1}, a_n, v_{sink})$ with $n \geq 1$ is called a resource constrained path of $G$ if
$v_i \in V$,
$a_i = (v_{i-1}, v_{i}) \in A$ (considering $v_0 = v_{source}$ and $v_n = v_{sink}$) and,  
for each~$r \in R$, the accumulated resource consumption at the $j$-th vertex of $p$, $S_{j,r} = \max\{l_{a_j,r}, S_{j-1,r} + q_{a_j,r}\}$ and is at most $u_{a_j,r}$ for all $1 \leq j \leq n$, with $S_{0,r} = 0$. Let $P$ be the set of all resource constrained paths of $G$. The objective is to find a resource constrained path $p \in P$ that minimizes the sum of the costs of the arcs used.

Let $p$ be any resource constrained path of $G$, $h^p$ is a vector of integer variables that describes which arcs are used in $p$, so for any arc $a \in A$, $h^p_a$ represents the number of times arc $a$ is used in path $p$. This concept will be useful for mapping resource constrained paths to variables in the master formulation.

% As previously stated, the idea is to model our problem in such a way that the pricing problem is a RCSP\@. To illustrate this concept, consider an example for the BPP provided by \citet{pessoaSUV20}. 
% Let $G = (V,A)$, with $V = \{v_0, v_1, \dots, v_{|\I|}\}$
% and $A = \{ a_i^+, a_i^- \colon i \in \I \}$ where, for $i \in \I$,
% $a_i^+ = (v_{i-1}, v_i)$ and $a_i^- = (v_{i-1}, v_i)$ are parallel arcs. We also define 
% $R = \{r\}$,
% $v_{source} = v_0$, 
% $v_{sink} = v_{|\I|}$,
% $l_{v,r} = 0$ and $u_{v,r} = L$ for $v \in V$, and
% $q_{a_i^+,r} = l_i$ and $q_{a_i^-,r} = 0$ for $i \in \I$.

% The cost of the positive arcs are given by the dual values of the master formulation and the cost of the negative arcs are zero.\todo{Yulle, confirma para mim.} In the graph described above, a path from $0$ to $|\I|$ describes which items are included in a packing pattern. Let $p$ be a path in graph $G$, if $h^p_{a^+_i} > 0$ then item $i$ is included in~$p$. Conversely, if~$h^p_{a^-_i}>0$ then item $i$ is not included in the resulting pattern.
% \todo[inline]{Insert figure to illustrate the RCSP graph for the BPP}

\subsubsection{Modeling and Mappings}\label{sec:modeling}
The VRPsolver's restricted master problem formulation has three sets of variables. The first set of $n_1$ integer variables are mapped to the arcs of the RCSP graphs. The second set of $n_2$ integer variables are generic and allow modeling different constraints. Finally, there is a non-negative integer variable $\lambda_p$ for each path in the list of resource constrained paths $P$ considered in the restricted master problem.

For each variable of the first set, $x_j$ for $j \in [n_1]$, there must exist a \textit{mapping} $M(x_j)$ into a non-empty subset of arcs of the RCSP graph $G$. While not all arcs of $G$ need to be mapped into a variable of the model, it is still useful to define $M^{-1}(a) = \{j|a \in M(x_j)\}$.

Although the framework allows for multiple graphs to be used together for the same model, we omit it from the formulation for simplicity, since this fact is not relevant in our case. With this in mind, we can describe the master formulation as follows.
{\footnotesize
\begin{alignat}{4}
    \quad GMF=\,& \omit\rlap{$\displaystyle \min\sum_{j=1}^{n_1} c_jx_j + \sum_{s=1}^{n_2} f_sy_s$}                                                                                               \label{obj:gmf}\\
                & \mbox{s.t.}   &&\quad    &  \sum_{j=1}^{n_1}\alpha_{ij}x_{j} + \sum_{s=1}^{n_2}\beta_{is}x_{is}   & \geq d_i          \qquad && \forall\, 1 \leq i \leq m,                       \label{constr:gmf:gen}\\ 
                &               &&         &  x_{j} = \sum_{p \in P}\Bigg(\sum_{a \in M(x_j)} h^p_a \Bigg)          & \phantom{\leq}\lambda_p    \qquad && \forall\, 1 \leq j \leq n_1,            \label{constr:gmf:map} \\
                &               &&         &  \mathcal{L} \leq \sum_{p \in P}\lambda_p                                        & \leq \mathcal{U},                                                                      \label{constr:gmf:limits} \\
                &               &&         &  \lambda_p                                                             & \in \mathbb{Z}^+  \qquad && \forall p \in P,                                 \label{constr:gmf:lambda} \\
                &               &&         &  x_j, y_s                                                              & \in \mathbb{Z}    \qquad && \forall\, 1 \leq j \leq n_1 , 1 \leq s \leq n_2, \label{constr:gmf:integrality}
\end{alignat}
}%
in which Equation~\eqref{obj:gmf} is a general linear objective function with $c \in \mathbb{R}^{n_1}$ and $f \in \mathbb{R}^{n_2}$ being cost vectors, Constraints~\eqref{constr:gmf:gen} represents $m$ general linear constraints over the first two sets of variables, with $\alpha \in \mathbb{R}^{m\times n_1}$ and $\beta \in \mathbb{R}^{m\times n_2}$ being the coefficient matrices for these variables, and $d \in \mathbb{R}^m$ the right-hand side vector. Constraints~\eqref{constr:gmf:map} are the mapping constraints, and ties variables of the first group to arcs in the RCSP graph. Constraints~\eqref{constr:gmf:limits} define lower bounds $\mathcal{L}$ and upper bounds $\mathcal{U}$ on the number of paths used in the solution. Constraints~\eqref{constr:gmf:lambda}--\eqref{constr:gmf:integrality} define the domains of variables.

With a specifically designed RCSP graph and the generic master formulation above, we can define specific models for several problems in the literature.

\subsubsection{Packing Sets}
The VRPSolver framework generalizes many concepts used in state-of-the-art branch-and-cut-and-price algorithms. Such concepts include ng-ranks, path enumeration, limited memory rank-1 cuts, Ryan-Foster branching rule, branching over accumulated resource consumption, and others. This is achieved with the introduction of \textit{Packing Sets}.

Packing Sets of the VRPSolver model are defined as a collection $\mathcal{P}$ of mutually disjoint subsets of arcs $A$, such that the arcs in each $S \in \mathcal{P}$ appear at most once in all paths that are part of some optimal solution $(x^*, y^*, \lambda^*)$ of the master problem, that is, for all $S \in \mathcal{P}$, $\sum_{p \in P} \Big(\sum_{a \in S} h^p_a\Big) \lambda_p^* \leq 1$. We refer to each element in $\mathcal{P}$ as a packing set.

The definition of the Packing Sets does not derive directly from the generic model, thus it is a modeling task to properly define it and show its validity. Furthermore, some features of the algorithm depend on the definition of the packing sets, and since some of these features such as the Ryan-Foster branching rule are necessary for the correctness of the model, the definition of packing sets becomes mandatory. We refer to~\citet{pessoaSUV20} for more information on Packing Sets and how they are used to generalize important algorithmic concepts of this framework. 

% To build on the model for the BPP presented in Section~\ref{sec:modeling}, we define its packing sets as $\mathcal{P} = \cup_{i \in \I} \{ \{a^+_i\} \}$. In other words, there is one packing set for each item in the instance, each including only the arc mapped to the variables of the master formulation for that item. Since the item lenghts are non-negative, there must exist an optimal solution in each item will be used exactly once, hence the packing sets are valid.

\subsection{Color-Resources Partition Model for the CBPP}\label{sec-exp-crsp}
For our first VRPSolver model for the CBPP, we propose a generalization of the model and graph described in the previous sections. Thus, consider a graph similar to the one presented by \citet{pessoaSUV20}, that is $G = (V,A)$, with $V = \{v_0, v_1, \dots, v_{|\I|}\}$ with $v_{source} = v_0$ and $v_{sink} = v_{|\I|}$, 
and $A = \{a_i^+, a_i^- \colon i \in \I \}$ where
$a_i^+ = (v_{i-1},v_i)$ and $a_i^- = (v_{i-1},v_i)$ 
for $i \in \I$.
%\todo[inline]{provide figure showing an example graph for this model, or perhaps figure from last section suffices}

As in the model for the BPP consider a resource $\bar r \in R$ for the capacity with $q_{a_i^+,\bar{r}} = l_i$ and $q_{a_i^-,\bar{r}} = 0$ for $i \in \I$ and $l_{v,\bar{r}} = 0$ and $u_{v,\bar{r}} = L$ for $v \in V$.  Consider also, for each color $c \in [Q]$, a resource $r_c \in R$ with $q_{a_i^-,r_c} = 0$ and, if $c_{i} = c$, then $q_{a_i^+,r_c} = 1$, and, if $c_{i} \neq c$, then $q_{a_i^+,r_c} = -1$. Finally, for each color $c \in [Q]$ and vertex $v \in V$, the bounds on the accumulated resources are $l_{v,r_c} = -|\I|$, and $u_{v,r_c} = 1$ if $v = v_{sink}$ and $u_{v,r_c} = |\I|$ otherwise. Like in the BPP model, the packing sets are the singletons composed of the positive arc of each item, that is $\mathcal{P} = \cup_{i \in \I} \{ \{ a_i^+ \} \}$. 

%Unlike with the previous graph though, consider one main resource for the capacity $\bar r \in R$ and one additional secondary resource $r_c \in R$ for each color $1 \leq c \leq Q$. For the capacity resource we have 
%\begin{align}
%    \begin{array}{ll} q_{a_i^+,\bar{r}} = l_i \\ q_{a_i^-,\bar{r}} = 0 \end{array}   & & \begin{array}{ll} \mbox{ for } 0 \leq i \leq |\I|, \\ \end{array}
%\end{align}%
%and $[l_{v_i,\bar{r}}, u_{v_i,\bar{r}}] = [0, L]$. For each color resource $r_c$, for $1 \leq c \leq Q$, we have
%\begin{align}
%    \begin{array}{ll} q_{a_i^+,r_c} = 1 \\ q_{a_i^+,r_c} = -1 \end{array}   & \begin{array}{ll} \mbox { if } c_i = c \\ \mbox{ Otherwise} \end{array} & \begin{array}{ll} \mbox{ for } 1 \leq i \leq |\I|, \\ \end{array}
%\end{align}%
%and $q_{a_i^-,r_c} = 0$, with bounds on the accumulated resources $[l_{v_i,r_c}, u_{v_i,r_c}] = [-|\I|,|\I|]$ for all $v_i \in V\setminus \{v_{sink}\}$ and $[l_{v_{sink},r_c}, u_{v_{sink},r_c} = [-|\I|,1]$ for all $1 \leq c \leq Q$. Like in the BPP model, the packing sets are the singletons composed of the positive arc of each item, that is $\mathcal{P} = \cup_{1 \leq i \leq |\I|} \{ \{ a_i^+ \}\}$. 

In this model, all color constraints are enforced in the pricing problem through the RCSP graph. Whenever an item of color $c$ is included in a resource constrained path of $G$, one unit of resource $r_c$ is collected and one unit of resource is dropped for every other $r_{c'}$ such that $c' \neq c$. Through this procedure, we have the color discrepancy for the path at each vertex, so we enforce that such color discrepancy be smaller than or equal to $1$ at the sink vertex for all colors. From Theorem~\ref{theo:definition}, we know that this condition is necessary and sufficient to guarantee the existence of a color-alternating permutation of the items in the path. Since all color constraints are modeled in the RCSP, the master formulation can be described with integer variables $x_0, x_1,\dots, x_{|\I|}$, where $x_0$ represents the number of paths used, while~$x_i$ indicates whether item $i$ is covered in a path, 
and mappings $M(x_0) = \{a_1^+, a_1^-\}$, $M(x_i) = \{a_i^+\}$ for $i \in \I$. The master formulation is given by
{\footnotesize
\begin{alignat}{4}
    \quad  & \omit\rlap{$\displaystyle \min x_0$} \label{obj:cbpp1}\\
           & \mbox{s.t.}   &&\quad    &  x_{i}                                                                 &=1                          \qquad && \forall\, i \in \I',        \label{constr:cbpp1:demand}\\ 
           &               &&         &  x_{i} = \sum_{p \in P}\Bigg(\sum_{a \in M(x_i)} h^p_a \Bigg)          & \phantom{\leq}\lambda_p    \qquad && \forall\, i \in \I',        \label{constr:cbpp1:map} \\
           &               &&         &  \mathcal{L} \leq \sum_{p \in P}\lambda_p                                        & \leq \mathcal{U},                                                                   \label{constr:cbpp1:limits} \\
           &               &&         &  \lambda_p,x_{i}                                                       & \in \mathbb{Z}^+           \qquad && \forall\, p \in P, \forall i \in \I' \cup \{0\},\label{constr:cbpp1:integrality}                  
\end{alignat}%
}%
in which Equation~\eqref{obj:cbpp1} takes the place of Equation~\eqref{obj:gmf}, and represents the objective of minimizing the number of paths or packing patterns used. Constraints~\eqref{constr:cbpp1:demand} takes the place of the general Constraint~\eqref{constr:gmf:gen}, and enforce that each item appear in exactly one path used in the solution. The remaining constraints,~\eqref{constr:cbpp1:map}--\eqref{constr:cbpp1:integrality} are derived directly from general master formulation. For the BPP, we can use $\mathcal{L} = 1$ and $\mathcal{U} = \sum_{i \in \I} d_i$ as trivial bounds on the number of paths. 

\subsection{Colored-Clusters Partition Model for the CBPP}\label{sec-exp-ccsp}
For our second model, we propose a denser graph, but with only one resource constraint. Let $G = (V,A)$ be our colored-clusters graph. The set of vertices is $V = \{v_{source}, v_{sink}\} \cup \{v_i \colon i \in \I' \} \cup \{v_c^{in}, v_c^{out}\colon c \in [Q]\}$. And the set of arcs is 
$A = \{a_i^{source}, a_i^{sink}, a_i^{in}, a_i^{out} \colon i \in \I' \} \cup \{a_c^{c'} = (v_c^{out}, v_{c'}^{in}) \colon c,c' \in [Q], c \neq c' \}$ 
where, for $i \in \I'$,
$a_i^{source} = (v_{source},v_i)$,
$a_i^{sink} = (v_i,v_{sink})$,
$a_i^{in} = (v_{c_i}^{in}, v_i)$, and
$a_i^{out} = (v_i, v_{c_i}^{out})$.

In this graph, there exists one vertex per item as well as one inbound and one outbound vertex for each color. Arcs connect items to their respective inbound colors and outbound colors to all their items. Furthermore, there are arcs connecting each inbound color to every other outbound color. At last, there are arcs connecting the source to each item vertex and from each item vertex to the sink. Hence, all paths in $G$ will alternate between item vertices and color vertices and, since there are only arcs connecting inbound vertices to outbound vertices of different colors, there can be no consecutive items from the same color in a path. Thus, all color constraints are modeled in the RCSP\@. 
%\todo[inline]{Provide figure with an example of this graph}

Since there is only one resource for the capacity $L$, so we omit the resource index in the consumption notation. Thus, we have, $q_{a_i^{source}} = q_{a_i^{in}} = l_i$, and $q_{a_i^{sink}} = q_{a_i^{out}} = 0$, for $i \in \I'$, and $q_{a_c^{c'}} = 0$ for $c, c' \in [Q]$ with $c \neq c'$.

Finally, the bounds on the accumulated resources are $l_v = 0$ and $u_v = L$ for all $v \in V$ and the packing sets are defined as $\mathcal{P} = \cup_{i \in \I'} \{ \{ a_i^{source}, a_i^{in} \} \}$. Since each item must be included exactly once in an optimal solution, we know that such a solution path will either include an arc from the source to the item vertex $v_i$ or an arc from the inbound color vertex $v^{in}_{c_i}$ to the item vertex $v_i$ for $i \in \I'$.

Since all color constraints are modeled in RCSP, the master formulation can, once again, be described with integer variables $x_0, x_1, \dots, x_{|\I|}$, and mappings $M(x_0) = \{a_1^{source}, a_2^{source}, \dots, a_{|\I|}^{source}\}$, and $M(x_i) = \{a_i^{source}, a_i^{in}\}$ for $i \in \I$. With this graph and mappings, we can use Formulation~\eqref{obj:cbpp1}--\eqref{constr:cbpp1:integrality} to solve the CBPP\@.

\section{Numerical Experiments\label{sec-experiments}}
In this section, we discuss numerical experiments used to evaluate all the proposed models. Since, to the best of our knowledge, there are no benchmarks for the CBPP in the literature, we propose our own set of instances. In Section~\ref{sec-experiments-benchmark}, we describe our proposed benchmark set, and in Section~\ref{sec-experiments-results} we discuss the experimental results. 

\subsection{Benchmark Instances}\label{sec-experiments-benchmark}
We propose three different sets of instances, with different levels of complexity. Each set is described below.

\paragraph{Uniform Randomly Generated Instances} It is composed of instances with varied item lengths and bin capacities, similar to those proposed by \citet{borgesMSX20} for a class constrained version of the BPP\@. Most instances in this set should be solved by all models without difficulty. This group of instances is represented by a tuple $(M,L,Q,W)$. The value of~$M$ corresponds to the number of items in the instance, and we consider $M \in \{300, 500\}$. The value of~$L$ corresponds to the capacity of bins, and we consider $L \in \{500, 750, 1000\}$. The value of~$Q$ corresponds to the number of different colors of items, and we consider~$Q~\in~\{2,7,15\}$. The value of $W$ corresponds to the range for the length of items relative to the bin capacity, and we consider $W \in  \{[0.1, 0.8], [0.01, 0.25]\}$. We generate $10$ instances for each tuple, resulting in 360 total instances.

\paragraph{Randomly Generated Instances with Zipf Distribution of Colors} It is composed of randomly generated instances in which the number of items of each color is uneven.  In this set, we have instances in which there are few colors with many items and many colors with only a few items. This characteristic is useful to model, for instance, the application of packing content in pages on a social media platform, in which we must alternate advertisements and other types of content while having most adverts come from only a few advertisers. The Zipf distribution~\citep{knuth98} has been used before to model data placement applications with such locality properties~\citep{xavierM08}. Instances in this set are represented by a tuple $(M,L)$. The value of $M$ represents the number of items, and $M \in \{300, 500\}$. The value of $L$ represents the capacity of the bins, and $L \in \{300, 500, 750\}$. The lengths of items are generated randomly with a uniform distribution in the interval of $[0.01, 0.25]$, relative to the bin capacity. All colors are drawn according to a Zipf distribution with $\alpha=2$. We generate $10$ instances for each tuple, totaling $60$ instances in this set. 

\paragraph{Instances Adapted from BPP Literature} It is composed of hard instances adapted from BPPLIB~\citep{delormeIM18}. This group is divided into $94$ instances in which the integer round-up property holds (AI), and $60$ instances in which the integer round-up property does not hold (ANI). We used only instances from which an optimal solution was known, all optimal solutions used were taken from \citet{delormeI20}. Given an AI or ANI instance, for each bin in its optimal solution, the items are sorted in a non-increasing order and colored the first half with one color and the second half with a different color. This would guarantee that the same configuration of items in that optimal solution would be feasible in the CBPP (by Theorem~\ref{theo:definition}), thus maintaining the same optimal solution value. 

\subsection{Computational Results}\label{sec-experiments-results}
We have implemented all the proposed models to evaluate their performance. The pseudo-polynomial models of Section~\ref{sec-pseudo} were implemented in C++11, using Gurobi version 9.03 for solving MILP models. The exponential models of Section~\ref{sec-exp} were implemented in Julia version~1.2, using VRPSolver beta v0.3 and CPLEX version 12.8 for solving the underlying LP models. Experiments were run on Intel (R) Xeon (R) CPU E5--2630 v4 with~10 cores at~2.20GHz, 64 GB of RAM, and running on Ubuntu 18.04.2 LTS with Linux~\mbox{4.15.0--45--generic}. 

For the pseudo-polynomial models, we have implemented a simple FF-like heuristic to obtain a simple initial feasible solution. For the exponential models, VRPSolver has a diving heuristic that is run periodically during the search in order to find feasible solutions. None of the models were provided with cutoff values to reduce the solution space. All models were solved with the single-thread mode of their respective solvers. We consider a time limit of~1800 seconds for each instance. 

\subsubsection{Results on Uniform Randomly Generated Instances}
We start by evaluating the proposed models on the set of Uniform Randomly Generated Instances. This is the largest of the proposed sets and covers a wide variety of cases. Since, for the CBPP, the color constraints are enforced on the adjacency of items, it makes sense to investigate instances in which many items can be packed in the same bin. For this reason, half of the instances in this set contain only items with lengths of at most one-quarter of the bin capacity, thus fitting at least four items in any maximal packing pattern. This set also tests the weight of having many different colors, since the number of colors ranges from~$2$ to~$15$. 

\begin{table}[ht!]
\scriptsize
\centering
\begin{tabular}{@{}cccccccccccccccccc@{}}
\cmidrule(r){1-6} \cmidrule(lr){8-10} \cmidrule(lr){12-14} \cmidrule(l){16-18}
    &      &    & \multicolumn{3}{c}{CA-AF} &  & \multicolumn{3}{c}{ML-AF} &  & \multicolumn{3}{c}{CC-partition} &  & \multicolumn{3}{c}{CR-partition} \\ \cmidrule(lr){4-6} \cmidrule(lr){8-10} \cmidrule(lr){12-14} \cmidrule(l){16-18} 
M   & L    & Q  & opt          & time & gap &  & opt          & time & gap &  & opt           & time   & gap     &  & opt           & time   & gap     \\ \midrule
300 & 500  & 2  & \textbf{10}  & 3    & 0   &  & \textbf{10}  & 4    & 0   &  & \textbf{10}   & 95     & 0       &  & \textbf{10}   & 42     & 0       \\
300 & 500  & 7  & \textbf{10}  & 9    & 0   &  & \textbf{10}  & 30   & 0   &  & \textbf{10}   & 146    & 0       &  & \textbf{10}   & 761    & 0       \\
300 & 500  & 15 & \textbf{10}  & 8    & 0   &  & \textbf{10}  & 105  & 0   &  & 9             & 341    & --      &  & \textbf{10}   & 1118   & 0       \\
300 & 750  & 2  & \textbf{10}  & 9    & 0   &  & \textbf{10}  & 6    & 0   &  & \textbf{10}   & 85     & 0       &  & \textbf{10}   & 82     & 0       \\
300 & 750  & 7  & \textbf{10}  & 9    & 0   &  & \textbf{10}  & 60   & 0   &  & 8             & 589    & --      &  & \textbf{10}   & 712    & 0       \\
300 & 750  & 15 & \textbf{10}  & 16   & 0   &  & \textbf{10}  & 235  & 0   &  & 9             & 419    & --      &  & \textbf{10}   & 1176   & 0       \\
300 & 1000 & 2  & \textbf{10}  & 8    & 0   &  & \textbf{10}  & 8    & 0   &  & \textbf{10}   & 205    & 0       &  & \textbf{10}   & 73     & 0       \\
300 & 1000 & 7  & \textbf{10}  & 16   & 0   &  & \textbf{10}  & 81   & 0   &  & \textbf{10}   & 251    & 0       &  & \textbf{10}   & 982    & 0       \\
300 & 1000 & 15 & \textbf{10}  & 20   & 0   &  & \textbf{10}  & 319  & 0   &  & 9             & 427    & 0.001   &  & \textbf{10}   & 1243   & 0       \\
500 & 500  & 2  & \textbf{10}  & 16   & 0   &  & \textbf{10}  & 10   & 0   &  & \textbf{10}   & 323    & 0       &  & \textbf{10}   & 193    & 0       \\
500 & 500  & 7  & \textbf{10}  & 13   & 0   &  & \textbf{10}  & 99   & 0   &  & 6             & 1053   & 0.002   &  & \textbf{10}   & 1204   & 0       \\
500 & 500  & 15 & \textbf{10}  & 20   & 0   &  & \textbf{10}  & 234  & 0   &  & 6             & 1167   & --      &  & 8             & 1460   & 0.001   \\
500 & 750  & 2  & \textbf{10}  & 23   & 0   &  & \textbf{10}  & 18   & 0   &  & \textbf{10}   & 593    & 0       &  & \textbf{10}   & 262    & 0       \\
500 & 750  & 7  & \textbf{10}  & 45   & 0   &  & \textbf{10}  & 140  & 0   &  & 8             & 938    & --      &  & 9             & 1400   & 0       \\
500 & 750  & 15 & \textbf{10}  & 44   & 0   &  & \textbf{10}  & 552  & 0   &  & 7             & 1166   & --      &  & 7             & 1721   & --      \\
500 & 1000 & 2  & \textbf{10}  & 25   & 0   &  & \textbf{10}  & 19   & 0   &  & \textbf{10}   & 573    & 0       &  & \textbf{10}   & 418    & 0       \\
500 & 1000 & 7  & \textbf{10}  & 43   & 0   &  & \textbf{10}  & 207  & 0   &  & 7             & 1060   & --      &  & 9             & 1530   & --      \\
500 & 1000 & 15 & \textbf{10}  & 54   & 0   &  & 9            & 947  & 0.1 &  & 9             & 555    & --      &  & 3             & 1770   & --      \\ \bottomrule
\end{tabular}
\caption{Results for all models on Uniform Randomly Generated Instances set with $W = [0.1, 0.8]$. Columns $M, L, Q$ represent the number of items in the instance, the bin capacity, and the number of colors respectively. For each of the proposed models, we have a column $opt$ showing the number of instances of that class (out of $10$) that were solved to proven optimality, $time$ which represents the average time in seconds, and $gap$ which represents the optimality gap computed as $(UB - LB)/UB$.}\label{tab:random-summary-1}
\end{table}

\begin{table}[ht]
\scriptsize
\centering
\begin{tabular}{@{}cccccccccccccccccc@{}}
\toprule
    &      &    & \multicolumn{3}{c}{CA-AF} &  & \multicolumn{3}{c}{ML-AF}  &  & \multicolumn{3}{c}{CC-partition} &  & \multicolumn{3}{c}{CR-partition} \\ \cmidrule(lr){4-6} \cmidrule(lr){8-10} \cmidrule(lr){12-14} \cmidrule(l){16-18} 
M   & L    & Q  & opt        & time & gap   &  & opt         & time & gap   &  & opt           & time    & gap    &  & opt            & time    & gap   \\ \midrule
300 & 500  & 2  & 9          & 576  & 0.053 &  & \textbf{10} & 733  & 0     &  & 0             & 1982    & --     &  & \textbf{10}    & 130     & 0     \\
300 & 500  & 7  & \textbf{4} & 1373 & 0.045 &  & 0           & 1801 & 0.508 &  & 0             & 1999    & --     &  & 0              & 1023    & --    \\
300 & 500  & 15 & \textbf{7} & 1117 & 0.039 &  & 0           & 1803 & 1     &  & 0             & 1934    & --     &  & 0              & 1024    & --    \\
300 & 750  & 2  & 8          & 1188 & 0.096 &  & 7           & 970  & 0.007 &  & 0             & 1991    & --     &  & \textbf{10}    & 195     & 0     \\
300 & 750  & 7  & \textbf{1} & 1698 & 0.096 &  & 0           & 1802 & 1     &  & 0             & 2012    & --     &  & 0              & 1034    & --    \\
300 & 750  & 15 & \textbf{0} & 1801 & 0.16  &  & \textbf{0}  & 1804 & 1     &  & \textbf{0}    & 1920    & --     &  & \textbf{0}     & 1037    & --    \\
300 & 1000 & 2  & 3          & 1557 & 0.304 &  & 9           & 964  & 0.005 &  & 0             & 1918    & --     &  & \textbf{10}    & 384     & 0     \\
300 & 1000 & 7  & \textbf{1} & 1791 & 0.234 &  & 0           & 1803 & 1     &  & 0             & 2033    & --     &  & 0              & 1045    & --    \\
300 & 1000 & 15 & \textbf{2} & 1732 & 0.223 &  & 0           & 1806 & 1     &  & 0             & 1901    & --     &  & 0              & 1048    & --    \\
500 & 500  & 2  & 5          & 1387 & 0.219 &  & 8           & 926  & 0.003 &  & 0             & 1847    & --     &  & \textbf{10}    & 462     & 0     \\
500 & 500  & 7  & \textbf{6} & 1457 & 0.056 &  & 0           & 1802 & 1     &  & 0             & 1862    & --     &  & 0              & 1069    & --    \\
500 & 500  & 15 & \textbf{5} & 1535 & 0.042 &  & 0           & 1805 & 1     &  & 0             & 1835    & --     &  & 0              & 1071    & --    \\
500 & 750  & 2  & 0          & 1801 & 0.409 &  & 5           & 1486 & 0.012 &  & 0             & 1826    & --     &  & \textbf{10}    & 523     & 0     \\
500 & 750  & 7  & \textbf{0} & 1801 & 0.307 &  & \textbf{0}  & 1803 & 1     &  & \textbf{0}    & 1851    & --     &  & \textbf{0}     & 1094    & --    \\
500 & 750  & 15 & \textbf{0} & 1802 & 0.282 &  & \textbf{0}  & 1808 & 1     &  & \textbf{0}    & 1821    & --     &  & \textbf{0}     & 1101    & --    \\
500 & 1000 & 2  & 0          & 1801 & 0.46  &  & 0           & 1801 & 0.046 &  & 0             & 1829    & --     &  & \textbf{9}     & 834     & --    \\
500 & 1000 & 7  & \textbf{0} & 1802 & 0.316 &  & \textbf{0}  & 1805 & 1     &  & \textbf{0}    & 1835    & --     &  & \textbf{0}     & 1134    & --    \\
500 & 1000 & 15 & \textbf{0} & 1803 & 0.314 &  & \textbf{0}  & 1810 & 1     &  & \textbf{0}    & 1821    & --     &  & \textbf{0}     & 1138    & --    \\ \bottomrule
\end{tabular}
\caption{Results for all models on Uniform Randomly Generated Instances set with $W = [0.01, 0.25]$. Captions in this table should be interpreted as in Table~\ref{tab:random-summary-1}.\label{tab:random-summary-2}}
\end{table}

Tables~\ref{tab:random-summary-1} and~\ref{tab:random-summary-2} show the results for all models on the first instance set. In Table~\ref{tab:random-summary-1} we see the results for instances with item length relative range $W = [0.1, 0.8]$, while Table~\ref{tab:random-summary-2} has the results of all instances with item length relative range $W = [0.01, 0.25]$.  

The color-alternating arc flow model of Section~\ref{sec-pseudo-caaf} (hereafter referred to as CA-AF) shows good results for this instance set, optimally solving $231$ out of the $360$ instances ($65\%$) with a relatively low average gap on the classes it does not solve completely. All instances with $W = [0.1, 0.8]$ are solved to optimality in under a minute on average. However, the results are not as impressive when it comes to instances with $W = [0.01, 0.25]$. Only $51$ out of $180$ ($28\%$) of these instances are solved to optimality, with average times coming closer to the time limit even on classes where half of the instances are solved optimally. This discrepancy can be explained by the fact that the number of arcs in the graph increases dramatically with the number of items that can be packed in a single bin, resulting in a model with many more variables in cases with $W = [0.01, 0.25]$. The largest instance in this subset to be solved to optimality by this model has $500$ items, $500$ of bin capacity, and~$15$ colors.  

Comparatively, the multilayered arc flow model of Section~\ref{sec-pseudo-mlaf} (hereafter referred to as ML-AF) shows milder results. Only $218$ out of the $360$ instances (60\%) are solved optimally while failing to even find a lower bound before the time limit in several of the instances it does not solve. When considering only instances with $W = [0.1, 0.8]$, all but one are solved to optimality in a few hundred seconds on average. On the other hand, for instances with~$W = [0.01, 0.25]$, the model was able to solve only $39$ out of $180$ ($21\%$). Also, for this subset, the model is only able to reliably find lower bounds for instances with $2$ colors. This can be explained by the fact that the number of vertices in the graph, for this model, grows by an order of $M$ for every color considered, making it difficult to even solve the root node of the MILP for instances with a large enough number of vertices and colors. The largest instance in this subset to be solved to optimality by this model has $500$ items, $750$ of bin capacity, and $2$ colors. Since this model has fewer constraints than the previous one, it may outperform CA-AF in specific cases where the number of colors is small enough, which can be observed in instances with $Q=2$.

For the colored-clusters partition model of Section~\ref{sec-exp-ccsp} (hereafter referred to as CC-partition) we have the worst results out of the models tested in this set. A total of $158$ out of $360$ instances ($43\%$) were optimally solved, with none of them being from the subset with~$W = [0.01, 0.25]$. Even for the subset with $W = [0.1, 0.8]$, it was not able to find a feasible solution for many cases, resulting in the impossibility of even computing an optimality gap (shown as ``--'' in the tables). Since this is a pattern-based model, the inability to solve instances with $W=[0.01, 0.25]$ can be explained by the fact that the number of possible patterns considered is comparatively higher for these cases. Furthermore, each pattern is generated as a resource constrained path, which also grows with the number of items that can be accommodated in the same bin, making the column generation phase more difficult in these cases. 

On the other hand, the color-resources partition model of Section~\ref{sec-exp-crsp} (CR-partition) has shown to be very competitive. It was able to solve $225$ out of $360$ instances ($62\%$) to optimality, although with a considerably higher average time. In the subset with $W = [0.1, 0.8]$, $166$ out of $180$ were optimally solved while failing to find upper bounds for some of the larger instances. Similarly to ML-AF, this model seems to perform its best in instances with only a few colors, as observed in cases with $Q=2$. In fact, for the subset with~$W=~[0.01,~0.25]$, all but one of the instances with $Q=2$ were solved to optimality. For instances with $Q>2$ in this subset, the model was interrupted before the time limit because of memory constraints, so it was not able to produce an optimality gap. This is because the color constraints are enforced as resource constraints in the column generation phase, instead of being represented as special vertices in the path like in CC-partition. Though this may result in a relative performance gain for a smaller number of colors, models with a higher number of resource constraints tend to be significantly more difficult. Furthermore, since branching on accumulated resource consumption is used, this may result in a very large branch-and-bound tree, possibly causing memory issues for large enough instances. The largest instance in this subset to be optimally solved has $500$ items, $1000$ of bin capacity, and $2$ colors. 

\subsubsection{Results on Randomly Generated Instances with Zipf Distribution of Colors}
We evaluate the proposed models on the set of Randomly Generated Instances with Zipf Distribution of Colors. In this relatively smaller instance set, we investigate the impact of the color constraints in the models by presenting cases with an unbalanced number of items for each color as well as a high number of total colors (up to 37).  In this set, all instances have lengths in the interval of $[0.1, 0.8]$ relative to the bin capacity, so maximal packings tend to contain several items. Since many items can be of the same color, it can be extremely difficult to find efficient packings.  

\begin{table}[ht!]
\centering
\begin{tabular}{cccclccclc}
\toprule
    &     &  & \multicolumn{3}{c}{CA-AF}                      &  & \multicolumn{3}{c}{ML-AF}              \\ \cline{4-6} \cline{8-10} 
M   & L   &  & opt         & \multicolumn{1}{c}{time} & gap   &  & opt & \multicolumn{1}{c}{time} & gap   \\ \midrule
300 & 300 &  & \textbf{8}  & 509                      & 0.049 &  & 3   & 1732                     & 0.700 \\
300 & 500 &  & \textbf{9}  & 539                      & 0.060 &  & 0   & 1805                     & 1     \\
300 & 750 &  & \textbf{9}  & 897                      & 0.002 &  & 0   & 1807                     & 1     \\
500 & 300 &  & \textbf{10} & 176                      & 0     &  & 0   & 1806                     & 0.941 \\
500 & 500 &  & \textbf{8}  & 1022                     & 0.111 &  & 0   & 1810                     & 1     \\
500 & 750 &  & \textbf{4}  & 1665                     & 0.291 &  & 0   & 1816                     & 1     \\ \bottomrule
\end{tabular}
\caption{Results for arc flow models on Randomly Generated Instances with Zipf Distribution of Colors set. Column $M$ represents the number of items in the instance, and $L$ is the bin capacity considered. Each model has a set of columns $opt$, $time$, and $gap$ meaning, respectively, amount of instances to be solved optimally (out of $10$), average time in seconds, and average gap computed as $(UB - LB)/UB$.}\label{tab:zipf-summary}
\end{table}

Table~\ref{tab:zipf-summary} shows a summary of the results for this set. Following a trend observed in the previous instance set, model CC-partition failed to solve any of the instances in this set, so we decided to omit it from the reports. Due to the high number of colors considered in all instances of this set, model CR-partition failed to even load most of them due to memory constraints, so this model is also omitted from the reports. 

Model CA-AF was able to optimally solve $48$ out of $60$ instances ($80\%$) in this set, while ML-AF was only able to solve $3$ to optimality ($0.05\%$). This was expected since we knew from previous analysis that ML-AF could only surpass CA-AF in cases with a very small number of colors. We can also notice that despite the significantly larger number of colors, CA-AF was faster on average and solved more instances when compared to similar instances in the Uniform Randomly Generated Instances set. This shows that, unlike any of the other proposed models, CA-AF is reasonably efficient in dealing with high-degree color constraints. 

\subsubsection{Results on Instances Adapted from BPP Literature}
We evaluate the proposed models on the set of Instances Adapted from BPP Literature. In this set, we adapt difficult instances from the BPP literature, namely the AI and ANI instances, to instances of the CBPP\@. These instances have specific properties to their lower bounds while solved by a classic set-cover formulation, so to maintain these properties, we guarantee that the optimal solution value is maintained when adapting them to the CBPP\@. Additionally, we consider only two colors in the adaptation, so all models can be competitive. For more details on the original instances and their properties, we refer to~\citet{delormeIM18}. Table~\ref{tab:aiani-summary} shows the results for all models in this instance set. 

\begin{table}[ht!]
\scriptsize
\centering
\begin{tabular}{@{}ccccccccccccccccc@{}}
\toprule
                &    & \multicolumn{3}{c}{CA-AF} &  & \multicolumn{3}{c}{ML-AF} &  & \multicolumn{3}{c}{CC-partition} &  & \multicolumn{3}{c}{CR-partition} \\ \cmidrule(lr){3-5} \cmidrule(lr){7-9} \cmidrule(lr){11-13} \cmidrule(l){15-17} 
instance        & \# & opt    & time   & gap     &  & opt    & time   & gap     &  & opt           & time   & gap     &  & opt           & time   & gap     \\ \midrule
201\_2500\_AI   & 50 & 42     & 951    & 0.003   &  & 39     & 1099   & 0.023   &  & 45            & 392    & 0.002   &  & \textbf{46}   & 386    & 0.001   \\
402\_10000\_AI  & 36 & 0      & 1805   & 1       &  & 0      & 1804   & 1       &  & 25            & 994    & --      &  & \textbf{32}   & 1105   & --      \\
600\_20000\_AI  & 8  & 0      & 1816   & 1       &  & 0      & 1813   & 1       &  & \textbf{2}    & 1759   & --      &  & 0             & 2271   & --      \\
201\_2500\_ANI  & 49 & 39     & 846    & 0.005   &  & 44     & 791    & 0.042   &  & 41            & 784    & --      &  & \textbf{48}   & 471    & 0       \\
402\_10000\_ANI & 11 & 0      & 1804   & 1       &  & 0      & 1804   & 1       &  & \textbf{1}    & 1772   & --      &  & 0             & 1810   & --      \\ \bottomrule
\end{tabular}
\caption{Results for all models on Instances Adapted from BPP Literature. Column $instance$ describes the original instance class, with the first number being an approximation of the number of items, and the second number an approximation of the bin capacity. Column $\#$ shows the number of instances for each class. Each model has a set of columns $opt$, $time$, and $gap$ meaning, respectively, amount of instances to be solved optimally, the average time in seconds, and the average gap computed as $(UB - LB)/UB$.}\label{tab:aiani-summary}
\end{table}

Model CA-AF was only able to solve the smaller instances of each set, being $42$ out of $94$ ($44\%$) from AI and $39$ out of $60$ ($65\%$) from ANI\@. The model was unable to even compute a lower bound for any instance of size greater than $200$. This was expected since the AI and ANI sets are very difficult for simpler arc flow models. Advanced reduction techniques are often required to deal with such instances in the BPP, such as the reflect model of~\citet{delormeI20} or the meet-in-the-middle model of \citet{coteI18}. Because of the nature of the color constraints, such techniques cannot be trivially adapted to the CBPP\@.

Model ML-AF does slightly better than CA-AF, solving $39$ out of $94$ instances ($41\%$) from AI and $44$ out of $60$ ($73\%$) from ANI\@. Similarly to the previous model, only instances of size $200$ were solved. This result was expected, due to ML-AF outperforming CA-AF in cases with only a few colors in previous experiments. Since this set only has instances with~$2$ colors, ML-AF presents a smaller number of constraints when compared to CA-AF, which may explain the slightly better results. 

We expected the exponential models to perform very well in this set, due to previous results of VRPSolver BPP models \citep{pessoaSUV20} even surpassing those of \citet{delormeI20}. This was confirmed, with both partition models performing similarly well. CC-partition solved $72$ out of $94$ instances ($76\%$) from AI and $42$ out of $60$ ($70\%$) from ANI, while CR-partition solved $78$ out of $94$ instances ($82\%$) from AI and $48$ out of $60$ ($80\%$) from ANI\@. Interestingly, although model CR-partition was able to solve more instances in total, model CC-partition was the only one to solve $2$ instances from AI with size $600$ and one instance from ANI with size $400$. This shows us that both these models have a high potential for solving difficult instances for cases when the number of colors is $2$.

\section{Concluding Remarks\label{sec-conclusions}}
In this work, we study the Colored Bin Packing Problem under an exact approach. We start by proposing a generalization of the arc flow model of \citet{carvalho99}, using multiple layers to represent different colors (ML-AF). Then, we propose a novel arc flow formulation with color-alternating constraints (CA-AF) and demonstrate that the additional constraints are sufficient to model the CBPP without having to use one layer per color. Furthermore, we prove that the two arc flow models are equivalent in terms of linear relaxation strength. We also present two set-partition models (CC-partition and CR-partition) to be solved by a branch-cut-and-price algorithm through a reduction to a general Vehicle Routing Problem with VRPSolver. A diverse benchmark set is also proposed to evaluate the models.   

Experimental results comparing all the proposed models are presented. Among the ones considered, the CA-AF can solve the most instances, the largest of which has $500$ items and~$37$ colors. However, when only two colors are considered, the set-partition models are shown to outperform the other models. 

It would be interesting to investigate in the future, the use of the characterization of feasible solutions given in Theorem~\ref{theo:definition} to design efficient heuristics and meta-heuristics for the CBPP\@. This characterization can also pave the way for new approximation algorithms for the CBPP or even approximation schemes for a colored variant of the knapsack problem. Good algorithms for the colored variant of the knapsack problem may lead to a good branch-and-price algorithm to solve the set-partition model based on the formulation of \citet{gilmoreG61}.

\section*{Acknowledgments}
This project was supported by the S\~ao Paulo Research Foundation (FAPESP) grants
\mbox{\#2015/11937--9} and     % temático fkm
\mbox{\#2022/05803--3},        % temático do Reinaldo recém aprovado 
% \mbox{\#2016/23552--7};      % regular rafael
%
and the Brazilian National Council for Scientific and Technological Development (CNPq) grants
%
%\mbox{\#425340/2016--3},       % universal fkm
%\mbox{\#308689/2017--8},       % pq antiga rafael
\mbox{\#144257/2019--0},       % bolsa yulle
\mbox{\#311039/2020--0} and    % pq rafael
\mbox{\#313146/2022--5}.       % pq fkm
This study was financed in part by the Coordena\c{c}\~ao de Aperfei\c{c}oamento de Pessoal de N\'\i{v}el Superior~-~Brasil (CAPES)~-~Finance Code 001.

\section*{Declarations of Interest} 
Yulle G. F. Borges: None; Rafael C. S. Schouery: None; Flávio K. Miyazawa: None.

\bibliography{../ref.bib}

\end{document}